\documentclass[a4paper,10pt,leqno]{amsart}
        \title[Algebraic $K$-theory over the infinite dihedral group]
              {Algebraic $K$-theory over the infinite dihedral group: a controlled topology approach}
       \author{James F. Davis}
       \address{Department of Mathematics\\
       Indiana University\\
       Bloomington, IN 47405, USA}
       \email{jfdavis@indiana.edu}
       
       \author{Frank Quinn}
       \address{Department of Mathematics\\
       Virginia Polytechnic Institute and State University\\
       Blacksburg, VA 24061, USA}
       \email{quinn@math.vt.edu}
       
       \author{Holger Reich}
      \address{Mathematisches Institut\\
               Freie Universit{\"a}t Berlin\\
               Arnimallee 7,
               D-14195 Berlin, Germany}
        \email{holger.reich@fu-berlin.de}
     
    \subjclass[2000]{}

\usepackage{calc}
\usepackage{amssymb,amsmath,amscd,amsthm,pb-diagram,enumerate,enumitem,tikz}
\usetikzlibrary{patterns,arrows,decorations.pathmorphing}
\usepackage{graphicx}
\usepackage[arrow,curve,matrix,tips,2cell]{xy}
  \SelectTips{eu}{10} \UseTips
  \UseAllTwocells

\DeclareMathAlphabet{\matheurm}{U}{eur}{m}{n}

\newcommand{\op}{\text{op}}

\newcommand{\Ab}{\matheurm{Ab}}

\newcommand{\gcw}{\matheurm{GroupCW}}
\newcommand{\gcwp}{\matheurm{GroupCWPairs}}
\newcommand{\Groupoids}{{\matheurm{Groupoids}}}
\newcommand{\Or}{\matheurm{Or}}
\newcommand{\OrG}{\matheurm{Or}G}

\newcommand{\spaces}{\matheurm{Top}}
\newcommand{\Spectra}{\matheurm{Spectra}}

\DeclareMathOperator{\aut}{aut}

\DeclareMathOperator{\hocolim}{{hocolim}}

\DeclareMathOperator{\id}{{id}}
\DeclareMathOperator{\inc}{inc}

\DeclareMathOperator{\map}{map}
\DeclareMathOperator{\mor}{mor}
\DeclareMathOperator{\obj}{obj}

\DeclareMathOperator{\proj}{Proj}
\DeclareMathOperator{\pt}{{pt}}
\DeclareMathOperator{\res}{res}

\DeclareMathOperator*{\sma}{\wedge}  

  \newcommand{\IN}{\mathbb{N}}

  \newcommand{\IR}{\mathbb{R}}

  \newcommand{\IZ}{\mathbb{Z}}

  \newcommand{\cala}{\mathcal{A}}
  
  \newcommand{\calc}{\mathcal{C}}
  \newcommand{\cald}{\mathcal{D}}

  \newcommand{\calg}{\mathcal{G}}
  \newcommand{\calh}{\mathcal{H}}

  \newcommand{\calu}{\mathcal{U}}
  \newcommand{\calv}{\mathcal{V}}

  \newcommand{\bfE}{{\mathbf E}}

  \newcommand{\bfH}{{\mathbf H}}

  \newcommand{\bfK}{{\mathbf K}}

  {\end{list}}

\newcommand{\bfnil}{{\mathbf {Nil}}}

\newcommand{\R}{{\mathbb R}}
\newcommand{\Z}{{\mathbb Z}}

\DeclareMathOperator{\nil}{Nil}

\DeclareMathOperator{\sub}{\matheurm{sub}}
\DeclareMathOperator{\fin}{\matheurm{fin}}

\DeclareMathOperator{\ff}{\matheurm{f}}
\DeclareMathOperator{\gggg}{\matheurm{g}}

\DeclareMathOperator{\famg}{\matheurm{g}}
\DeclareMathOperator{\famh}{\matheurm{h}}
\DeclareMathOperator{\fbc}{\matheurm{fbc}}
\DeclareMathOperator{\cyc}{\matheurm{cyc}}
\DeclareMathOperator{\vc}{\matheurm{vcyc}}

\DeclareMathOperator{\all}{\matheurm{all}}

\DeclareMathOperator{\Nil}{Nil}
\DeclareMathOperator{\WNil}{Nil}

\newcommand{\g}{\Gamma}
\newcommand{\go}{\Gamma_0}

\theoremstyle{plain}
\newtheorem{theorem}{Theorem}[section]

\newtheorem{lemma}[theorem]{Lemma}
\newtheorem{corollary}[theorem]{Corollary}

\theoremstyle{definition}
\newtheorem{definition}[theorem]{Definition}
\newtheorem{example}[theorem]{Example}

\newtheorem{remark}[theorem]{Remark}

\theoremstyle{remark}

\makeatletter\let\c@equation=\c@theorem\makeatother

\begin{document}

\begin{abstract}
We use controlled topology applied to the action of the infinite dihedral group on a partially compactified plane and deduce two consequences for algebraic $K$-theory.  The first is that 
the family in the $K$-theoretic Farrell-Jones conjecture can be reduced to only those
virtually cyclic groups which admit a surjection with finite kernel onto a cyclic group.
The second is that the Waldhausen Nil groups for a group which maps epimorphically onto the infinite dihedral group can be computed in terms of the Farrell-Bass Nil groups of the index two subgroup which maps surjectively to the infinite cyclic group.
\end{abstract}

\maketitle


\section{Introduction}

Let $G$ be a group.  Let $\OrG$ be its orbit category; objects are $G$-sets $G/H$ where $H$ is a subgroup of $G$ and morphisms are $G$-maps.  Let $R$ be a ring. Davis-L\"uck \cite{DL98} define a functor $\bfK_R : \Or G \to \Spectra$  with the key property  $\pi_n \bfK_R(G/H) = K_n(RH)$.  The utility of such a functor is to allow the definition of an equivariant homology theory, indeed for a $G$-CW-complex $X$, one defines
$$
H^G_n(X; \bfK_R) = \pi_n(\map_G(-,X)_+ \wedge_{\Or G} \bfK_R(-)),
$$
see \cite[section 4 and 7]{DL98} for basic properties.  Note that $\map_G(G/H,X) = X^H$ is the fixed point functor and that the ``coefficients'' of the homology theory are given by $H^G_n(G/H; \bfK_R) \cong K_n(RH)$.

A {\em family $\ff$ of subgroups of $G$} is a nonempty set of subgroups closed under subgroups and conjugation. For example,  the families
$$
1 \subset \fin \subset \fbc \subset \vc \subset \all
$$
consist of the trivial subgroup, the finite subgroups, the extensions
of finite by cyclic subgroups (groups that surject onto a cyclic group with finite kernel), the virtually cyclic subgroups (groups that have
a cyclic subgroup of finite index), and all
subgroups.  

For a family $\ff$, $E_{\ff}G$  is the classifying space for $G$-actions with isotropy in $\ff$.  
It is characterized up to $G$-homotopy equivalence as a $G$-CW-complex with $E_{\ff}G^H$  
contractible for subgroups $H \in \ff$ and $E_{\ff}G^H= \emptyset$ for subgroups $H \not \in \ff$.    
The Farrell-Jones isomorphism conjecture for algebraic $K$-theory \cite{FJ-isom} states that for every group $G$ and every ring $R$, the map
$$
H_n^G ( E_{\vc} G ; \bfK_{R}) \to H_n^G ( \pt ; \bfK_{R} ) =K_n( RG ).
$$
induced by the projection $E_{\vc} G \to \pt$ 
is an isomorphism.

Let $C_\infty$ be the infinite cyclic group,  $C_n$ the finite cyclic group of order $n$, and
$D_\infty =C_2 \ast C_2 \cong C_\infty \rtimes C_2$ the infinite dihedral group.  Virtually cyclic groups are either finite, surject to $C_\infty$ with finite kernel, or surject to $D_\infty$ with finite kernel, see \cite[Lemma~2.5]{FJ-lower} or \cite[Lemma~3.6]{DKR}.  The algebraic $K$-theory of groups surjecting to $C_\infty$ was partially analyzed by Farrell-Hsiang \cite{FH}; the algebraic $K$-theory of groups surjecting to $D_\infty$ was partially analyzed by Waldhausen \cite{Waldhausen(1978a), Waldhausen(1978b)}.  The trichotomy of virtually cyclic groups motivates a reexamination of the $K$-theory of groups surjecting to the infinite dihedral group.

For a group homomorphism $p \colon G \to H$
and a family $\ff$ of subgroups of $H$ define the pullback family by
\[
p^{\ast} \ff = \{ K \; | \; K \mbox{ is a subgroup of } G \mbox{ and } p(K) \in \ff \}.
\]

The following theorem is proved in Section~\ref{sec-proof-of-main}.

\begin{theorem} \label{thm-main}
Let $p \colon \g \to D_\infty$ be a surjective group homomorphism.
For every ring $R$ and every $n \in \IZ$ the map
\[
H_n^{\g} ( E_{p^{\ast} \fbc}\g ; \bfK_{R}) \to H_n^{\g} ( \pt ; \bfK_{R} ) = K_n( R\g )
\]
is an isomorphism.
\end{theorem}

\noindent
There is also a version of this theorem with coefficients, compare Remark~\ref{rem-coeff}.  Note that for the infinite dihedral group all finite-by-cyclic subgroups are cyclic, and that if $\ker p$ is finite, then $p^*\fbc = \fbc$.

The proof of Theorem~\ref{thm-main} in Section~\ref{sec-proof-of-main} below uses controlled topology and a geometric idea that is already implicit in
\cite{FJ-lower}. It is  a simple application of
\cite[Theorem~1.1]{Bartels-Lueck-Reich(hyperbolic)}.

Theorem \ref{thm-main} implies that the family of virtually cyclic subgroups can be replaced by the family of finite-by-cyclic subgroups in the Farrell-Jones conjecture and that the $K$-theory of a virtually cyclic group which surjects to $D_\infty$ can be computed in terms of the $K$-theory of its finite-by-cyclic subgroups.  More precisely, we sharpen the Farrell-Jones isomorphism conjecture (Corollary
\ref{cor-reduce-fj}) and compute certain Waldhausen Nil groups in terms of Farrell-Bass Nil groups (Theorem \ref{cor1}).  Both of these applications also follow from the main algebraic theorem of \cite{DKR}, but with different proofs.  Our Lemma \ref{dlnil} gives a translation between the controlled topology approach and the algebraic approach.

The transitivity principle (see \cite[Theorem A.10]{FJ-isom} or \cite[Theorem 65]{Lueck-Reich(survey)}) says that given families $\ff \subset \gggg$ of subgroups of $G$, if for all $H \in \gggg- \ff$, the assembly map
$$
H_n^H(E_{\ff \cap H}H; \bfK_R) \to H_n^H(\pt; \bfK_R)
$$
is an isomorphism, then the relative assembly map
$$
H_n^G(E_{\ff}G; \bfK_R) \to H_n^G(E_{\gggg}G; \bfK_R)
$$
is an isomorphism. Here $f \cap H = \{ K \; | \; K \in \ff, \; K \subset H \}$. As an immediate consequence of the transitivity principle and the classification trichotomy for virtually cyclic groups we obtain the following corollary of Theorem \ref{thm-main}.

\begin{corollary} \label{cor-reduce-fj}
For any group $G$ and ring $R$,
$$
H_n^G ( E_{\fbc} G ; \bfK_{R}) \to H_n^G (E_{\vc} G; \bfK_{R} )
$$
is an isomorphism.
\end{corollary}

There is an analogous statement for the fibered Farrell-Jones conjecture and the Farrell-Jones conjecture with coefficients, compare Remark~\ref{rem-coeff}.
The corollary implies that the Farrell-Jones conjecture in $K$-theory is equivalent to the conjecture that
$$
H_n^G ( E_{\fbc} G ; \bfK_{R}) \to H_n^G (\pt; \bfK_{R} ) = K_n(RG)
$$
is an isomorphism.

\begin{example} \label{ex-join-model}
We can realize $D_\infty$ as the subgroup of homeomorphism of the real line $\IR$ generated
by $x \mapsto x + 1$ and $x \mapsto -x$. Furthermore $D_\infty$ acts via the quotient map
$D_\infty \to D_\infty/C_\infty = C_2$ on $S^{\infty} = EC_2$. The join
\[
S^{\infty} \ast \IR
\]
is a model for $E_{\fbc}D_\infty$. This join-model was pointed out by Ian Hambleton.
See Lemma~\ref{lem-join} and Remark~\ref{rem-join} for a conceptional
explanation of this fact.
\end{example}

In order to explain the consequences of Theorem~\ref{thm-main} for Nil-groups
we introduce some more notation that will be used throughout the whole paper.
As above let $p \colon \g \to D_\infty=C_2\ast C_2$ be a surjective group homomorphism.
Let $C_\infty$ be the maximal infinite cyclic subgroup of $D_\infty$, let $\go = p^{-1} ( C_\infty )$, and let $F = \ker p$. Hence we have the following commutative diagram of groups with exact rows and columns.
\[
\xymatrix{
&& 1 \ar[d]   \ar[d]   & 1 \ar[d] &\\
1 \ar[r] &  F \ar@{=}[d]  \ar[r] & \go  \ar[d]  \ar[r]^-p  & C_\infty \ar[d] \ar[r] & 1 \\
1 \ar[r] &  F  \ar[r] & \g \ar[r]^-p \ar[d] & D_\infty \ar[d] \ar[r] & 1 \\
&& C_2 \ar[d] \ar@{=}[r] & C_2 \ar[d]  & \\
&& 1 & 1 &
}
\]
Both the groups $\go$ and $\g$ admit descriptions in terms of combinatorial group theory.  Choose $t \in \go$ so that $p(t) \in C_\infty$ is a generator.  This chooses a splitting of the epimorphism $\go \to C_\infty$ and hence expresses $\go$ as a semidirect product $\go = F \rtimes C_\infty$.
Moreover, let $G_1 = p^{-1} ( C_2 \ast 1 )$, $G_2 = p^{-1} ( 1 \ast C_2 )$.
We have the following two pushout-diagrams of groups. The left one maps via $p$ surjectively onto the right one.
\begin{eqnarray} \label{eq-pushout}
\xymatrix{
F \ar[d] \ar[r] & G_2 \ar[d] & 1 \ar[r] \ar[d] & C_2 \ar[d] \\
G_1 \ar[r] & \g = G_1 \ast_{F} G_2  & C_2 \ar[r] & D_\infty =  C_2 \ast C_2.
         }
\end{eqnarray}

Next note that $tRF \subset R\go$, $R[G_1 - F] \subset R\g$, and $R[G_2 - F] \subset R\g$ are $RF$-bimodules.  Given a ring $S$ and $S$-bimodules $M$, $N$, there are Waldhausen Nil groups $\widetilde{\nil}_n(S; M,N)$ and Farrell-Bass Nil groups $\widetilde{\nil}_n(S; M)$.  The definitions are reviewed in Section~\ref{sec-nil-rel}.  Lemma \ref{dlnil}(ii) below allows a translation between results from controlled topology that are formulated in terms of assmbly maps and results about Nil groups.  This translation lemma was already used by Lafont-Ortiz \cite{Lafont-Ortiz}.   It is key in proving our second main result.

\begin{theorem} \label{cor1}
Let $p \colon \g \to D_\infty$ be a surjective group homomorphism and let the notation be as above.
For all $n \in \IZ$ there is an isomorphism
\begin{eqnarray*}
 \widetilde{\WNil}_n ( RF ; R[G_1 - F] , R[G_2 - F] ) \cong
\widetilde{\Nil}_n (RF ; tRF ) .
\end{eqnarray*}
\end{theorem}

The Nil group on the left is the abelian group that measures failure of exactness of a Mayer-Vietoris type sequence for $K_*(R[G_1 \ast_{F} G_2 ])$ and two copies of the Nil group on the right  measures failure of exactness of a Wang type sequence for $K_*(R[F\rtimes C_\infty])$. One should think of the map from the right to the left in the above theorem as induced by the group inclusion $\go \subset \g$ and the map from left to right as induced by the two-fold transfer map associated to the inclusion.

\begin{remark} \label{rem-coeff}
In \cite{Bartels-Reich(coefficients)} the Farrell-Jones conjecture for $G$ with
coefficients in an additive category $\cala$ with $G$-action is developed.
This version with coefficients specializes to the Farrell-Jones conjecture
and also implies the fibered Farrell-Jones conjecture, i.e.;
the conjecture that for every surjective group homomorphism $p \colon \g \to G$ the map
\begin{eqnarray} \label{eq-fj-fibred}
H_n^{\g} ( E_{p^{\ast}\vc} \g ; \bfK_{R}) \to H_n^{\g} ( \pt ; \bfK_{R} ) =K_n( R\g )
\end{eqnarray}
is an isomorphism. Theorem~\ref{thm-main} holds also with coefficients, i.e; with
$\bfK_R$ replaced by $\bfK_{\cala}$ and $K_n ( R\g )$ replaced by $K_n ( \cala \ast_{\g} \pt )$.
Since the transitivity principle  is a fact about
equivariant homology theories, Corollary~\ref{cor-reduce-fj} also holds with coefficients
and therefore in the  fibered case, too.
\end{remark}

Corollary~\ref{cor-reduce-fj} and Theorem \ref{cor1} are given alternative proofs in \cite{DKR}.  
A difficult part of the story is to check that the approaches of this paper and of \cite{DKR} agree on their computational predictions.
Corollary~\ref{cor-reduce-fj} above is perhaps more natural from the point of view of this paper and Theorem \ref{cor1} is perhaps 
more natural from the point of view of \cite{DKR}.  \\

We would like to thank the referee for suggesting a simple   
proof of Lemma~\ref{spectra}. We would also like to thank the
National Science Foundation and the Instituto de Matematicas, Morelia
(IMUNAM) for their support and for funding the conference “Geometry,
Topology, and their Interactions” in 2007 in Morelia, Mexico where the
germination of this paper occurred and the Hausdorff Institute in
Bonn, Germany where it was finalized.

\section{Proof of Theorem~\ref{thm-main}} \label{sec-proof-of-main}

The version of Theorem~\ref{thm-main} with coefficients in an additive category, mentioned in Remark~\ref{rem-coeff}, follows
from Theorem \ref{thm-main-coeff-version} below by the inheritance properties proven in \cite[Corollary 4.3]{Bartels-Reich(coefficients)}.
Theorem~\ref{thm-main} itself then follows by specializing to the case where the additive category $\cala$ is
$R_{\oplus}$, i.e.; the category of finitely generated free $R$-modules equipped with the trivial $\Gamma$-action, compare \cite[Example 2.4]{Bartels-Reich(coefficients)}.

\begin{theorem} \label{thm-main-coeff-version}
Let $\cyc$ denote the family of cyclic subgroups of the infinite dihedral group $D_\infty$.
For every additive category $\cala$ with $D_\infty$-action and all $n \in \IZ$ the assembly map
\[
H_n^{D_\infty} ( E_{\cyc}D_\infty ; \bfK_{\cala}) \to H_n^{D_\infty} ( \pt ; \bfK_{\cala} ) =  K_n ( \cala \ast_{D_\infty} \pt )
\]
is an isomorphism.
\end{theorem}
Note that every subgroup of $D_{\infty}$  is either cyclic or infinite dihedral.
\begin{proof}
The proof will rely on Theorem~1.1 of \cite{Bartels-Lueck-Reich(hyperbolic)}.
According to that theorem the assembly map above is an isomorphism if we can show the following.
\begin{enumerate}
\item[(A)] \label{A}
There exists a $D_\infty$-space $X$ such that the underlying space $X$ is the realization of an abstract simplicial complex.
\item[(B)] \label{B}
There exists a $D_\infty$-space $\overline{X}$, which contains $X$ as an open $D_\infty$-subspace such that the underlying space
of $\overline{X}$ is compact, metrizable and contractible.
\item[(C)] \label{C}
There exists a homotopy $H \colon \overline{X} \times [0,1] \to \overline{X}$, such that $H_0 = \id_{\overline{X}}$
and $H_t(\overline{X}) \subset X$ for every $t >0$.
\item[(D)] \label{D}
There exists an $N \in \IN$ such that for all $\beta \geq 1$ there exists an open $\cyc$-cover $\calu ( \beta )$
of $D_\infty \times \overline{X}$ equipped with the diagonal $D_\infty$-action satisfying the following properties.
\begin{enumerate}
\item
For every $(g,x) \in D_\infty \times \overline{X}$ there exists a $U \in \calu( \beta )$, such that
\begin{eqnarray*} \label{eq-conditionDa}
B_{\beta}( g , d_w) \times \{ x \} \subset U.
\end{eqnarray*}
Here $B_{\beta} ( g , d_w ) \subset D_\infty$ denotes the ball of radius $\beta$ around $g \in D_\infty$ with respect to some fixed choice
of a word metric $d_w$ on $D_\infty$.
\item
The dimension of the nerve of the covering  $\calu( \beta )$
is smaller than $N$.
\end{enumerate}
\end{enumerate}
We recall that by definition an open $\cyc$-cover $\calu$ of a $\g$-space $Y$ is a collection of open
subsets of $Y$, such that the following conditions are satisfied.
\begin{enumerate}
\item $\bigcup_{U \in \calu} U = Y$.
\item If $g \in D_\infty$ and $U \in \calu$ then $g(U) \in \calu$.
\item If $g \in D_\infty$ and $U \in \calu$, then either $g(U) \cap U = \emptyset$ or $g ( U ) = U$.
\item For every $U \in \calu$ the group $\{ g \in D_\infty \; | \; g(U) = U \}$ lies in $\cyc$, i.e.; is a cyclic group.
\end{enumerate}

We realize the infinite dihedral group $D_\infty$ as the subgroup of the group of
homeo\-morphisms of the real line $\IR$ generated by
$s \colon x \mapsto x+1$ and $\tau \colon x \mapsto -x$. The action of $D_\infty$ on $\IR$ extends to an action
on $\overline{\IR} = [- \infty, \infty]$ by $s(\pm \infty) = \pm \infty$ and $\tau(\pm \infty) = \mp \infty$.

Clearly $X= \IR$ and $\overline{X} = \overline{\IR}$ satisfy the conditions (A),
(B) and (C). It hence suffices to find $\cyc$-covers of $D_\infty \times \overline{\IR}$ as required in (D).

We choose a point $x_0 \in \IR$, which lies in a free orbit of the $D_\infty$-action,
for example $x_0 = \frac{1}{4}$. We have the $D_\infty$-equivariant injective maps
\begin{eqnarray}
i \colon D_\infty \to \IR, \quad  g \mapsto gx_0 \quad \mbox{ and } \quad
j = i \times \id_{\overline{\IR}} \colon D_\infty \times \overline{\IR} \to \IR \times \overline{\IR} .
\end{eqnarray}
Let $d_w$ be the word metric on $D_\infty$
corresponding to the
generating system $\{ s^{\pm} , \tau \}$ and let $d$ denote the standard euclidean metric on $\IR$.
We have
$d( i(g) , i(g') ) \leq d_w( g , g') $
and hence
\[
i ( B_{\beta} (g, d_w) ) \subset B_{\beta} ( i(g) , d ).
\]
Therefore if $\beta \geq 1$ and $\calv( \beta)$ is a
$\cyc$-cover of $\IR \times \overline{\IR}$ such that
for every $(x,y) \in \IR \times \overline{\IR}$ there exists a $V \in \calv(\beta)$ such that
\begin{eqnarray} \label{condrr}
( x- \beta , x + \beta) \times \{ y \} \subset V,
\end{eqnarray}
then $\calu( \beta) = j^{-1} \calv(\beta) = \{ j^{-1} ( V ) \; | \; V \in \calu(\beta) \}$ is a $\cyc$-cover of $D_\infty \times \overline{\IR}$
whose dimension is not bigger than the one of $\calv(\beta)$ and it satisfies the condition (D)(a).

Two of the open sets in the desired $\cyc$-cover are given by
$V_+ = \{ (x,y) \in \IR \times \IR \; | \; x < y \} \cup \IR \times \{+ \infty \}$ and
$V_- = \{ (x,y) \in \IR \times \IR \; | \; y < x \} \cup \IR \times \{ - \infty \}$. Note that $\tau$ interchanges
$V_+$ and $V_-$ and that for the shift $s$ we have $s(V_+)=V_+$ and $s(V_-) = V_-$.

It remains to find a
$\cyc$-cover of a suitable neighbourhood (depending on $\beta$) of the diagonal $\Delta \subset \IR \times \IR \subset \IR \times \overline{\IR}$.
In a preparatory step we define a fixed $\cyc$-cover $\calv'$ of the diagonal $\Delta$ which is independent of $\beta$. Namely set
$z_0 = (0,0)$, $z_1= ( \frac{1}{2}, \frac{1}{2} )$ and define open balls in $\IR \times \IR \subset \IR \times \overline{\IR}$ by
\[
V_0 = \{ z \in \IR \times \IR \; | \; d(z_0 , z) < \frac{1}{2} \}, \quad
V_1= \{ z \in \IR \times \IR \; | \; d( z_1 , z ) < \frac{1}{2} \}.
\]


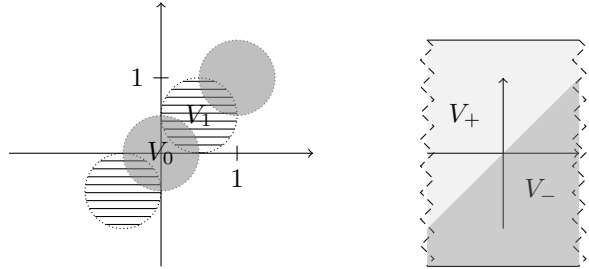
\begin{center} 
\begin{tikzpicture}[scale=.5]  [>=angle 90] 
\draw[->] (-2,-3) -- (-2,4);
\draw[->] (-6,0) -- (2,0);
\draw (0,.2) -- (0,-.2) node[below] {$1$};
\draw (-1.8,2) -- (-2.2,2) node[left] {$1$};
\draw[pattern=horizontal lines, densely dotted] (-1,1) circle (1cm);
\draw[pattern=horizontal lines, densely dotted] (-3,-1) circle (1cm);
\filldraw[fill=gray , opacity=.5,densely dotted] (-2,0) circle (1cm);
\filldraw[fill=gray , opacity=.5,densely dotted] (0,2) circle (1cm);
\draw (-2,0) node {$V_0$};
\draw (-1,1) node {$V_1$};

\draw (6,1) node {$V_+$};
\draw (8,-1) node {$V_-$};
\draw (5,3) -- (9,3);
\draw (5,-3) -- (9,-3);
\draw [densely dashed, decorate, decoration=zigzag] (5,3) -- (5,-3);
\draw[densely dashed, decorate, decoration=zigzag] (9,3) -- (9,-3);
\draw[->] (7,-2) -- (7,2);
\draw[->] (5,0) -- (9,0);
\fill[gray, opacity=.1] (5,-2) -- (9, 2) -- (9,3) -- (5,3) -- (5,-2) --cycle;
\fill[gray, opacity=.4] (5,-2) -- (9, 2) -- (9,-3) -- (5,-3) -- (5,-2) --cycle;
\end{tikzpicture}
\end{center}
\medskip

Then
\[
\calv' = \{ gV_0 \; | \; g \in C_{\infty} \} \cup \{ gV_1 \; | \; g \in C_{\infty} \}
\]
is a $\cyc$-cover (in fact even a $\fin$-cover) of the diagonal $\Delta$. Now
\[
\calv =  \{ V_+ , V_- \} \cup \calv'
\]
is a $\cyc$-cover of $\IR \times \overline{\IR}$ of dimension $2$. 

For all points $(x,y)$ with $y \in { \pm \infty}$ and arbitrary $\beta >0$
condition \eqref{condrr} is satisfied by choosing $V$ to be $V_+$ or $V_-$. There clearly exists an $\epsilon >0$, such that
for every $(x,y) \in \IR \times \IR$ there exists $V \in \calv$ such that
\begin{eqnarray} \label{eq-epsilon-wide}
(x - \epsilon, x + \epsilon) \times \{ y \}  \subset V.
\end{eqnarray}

For every $t \geq 0$ consider the homeomorphism
\[
\Phi_t \colon \IR \times \IR \to \IR \times \IR, \quad (x,y) \mapsto (x - t ( y-x) , y) .
\]
It is $D_\infty$-invariant with respect to the diagonal action,
preserves the horizontal lines $\IR \times \{ y \}$ and distances on such a line get stretched by a
factor of $1+t$, i.e.; if $d$ is the euclidean distance, then
\begin{eqnarray} \label{hor-stretch}
d (\Phi_t  (x,y) , \Phi_t(x' , y)) = (1+t) d(x,x').
\end{eqnarray}
The diagonal $\Delta \subset \IR \times \IR$ is fixed under $\Phi_t$.

For a given $\beta \geq 1$ choose $t = t(\beta) \geq 0$ such that $(1+t) \epsilon \geq \beta$ and set
\[
\calv ( \beta ) = \{ V_+ , V_- \} \cup \{  \Phi_{t(\beta)}( V )  \; | \; V \in \calv' \}.
\]
This is a $\cyc$-cover of $\IR \times \overline{\IR}$ of dimension $2$, which satisfies condition \eqref{condrr} as can be seen by combining \eqref{eq-epsilon-wide} and \eqref{hor-stretch}.
\end{proof}

\section{Nil-groups as relative homology groups} \label{sec-nil-rel}

The purpose of this section is to identify Nil-groups with relative homology
groups.  Lemma~\ref{dlnil} is folklore, but a proof has not explicitly appeared in the literature.  It is an interpretation of Waldhausen's  Theorem 1 and 3 from \cite{Waldhausen(1978a)} in the case of group rings.

We briefly explain the notation we use for the Nil-groups.  Let $M$ and $N$ be bimodules over a ring $S$.  Define an exact category
$\nil(S; M, N)$ whose objects are quadruples $(P,Q,p,q)$ where $P$ and $Q$ are finitely generated projective left $S$-modules and $p : P \to M \otimes_S Q$ and $q : Q \to N \otimes_S P$ are $S$-module maps subject to the nilpotence condition that
$$
P \xrightarrow{p} M \otimes_S Q \xrightarrow{1 \otimes q} M \otimes_S N  \otimes_S P  \xrightarrow{1 \otimes 1 \otimes p} M \otimes_S N  \otimes_S M \otimes_S Q \to \cdots
$$
is eventually zero.  Define $\nil_n(S;M,N) = \pi_n(\Omega BQ \nil(S;M,N))$ following Quillen \cite{Q} when $n \geq 0$ and define  $\nil_n(S;M,N)$ following Schlichting \cite{Sch-exact} when $n < 0$.  There is a split surjection of exact categories
\begin{align*}
\nil(S;M,N) &\to \proj(S) \times \proj(S)\\
(P,Q,p,q) & \mapsto (P,Q).\\
\end{align*}
Define groups $\widetilde{\nil}_*$ using this split surjection so that
$$
\nil_n(S;M,N) =\widetilde \nil_n(S;M,N) \oplus K_nS \oplus K_nS.
$$

There are similar definitions in the one-sided case which we now review.   Let $M$ be a bimodule over a ring $S$.  Define an exact category $\nil(S; M)$ whose objects are pairs $(P,p)$ where $P$ is finitely generated projective left $S$-module and $p : P \to M \otimes_S P$ is an $S$-module map subject to the nilpotence condition that
$$
P \xrightarrow{p} M \otimes_S P \xrightarrow{1 \otimes p} M \otimes_S M  \otimes_S P  \xrightarrow{1 \otimes 1 \otimes p} M \otimes_S M  \otimes_S M \otimes_S P \to \cdots
$$
is eventually zero.  Define $\nil_n(S;M) = \pi_n(\Omega BQ \nil(S;M))$ following Quillen \cite{Q} when $n \geq 0$ and define  $\nil_n(S;M)$ following Schlichting \cite{Sch-exact} when $n < 0$.  There is a split surjection of exact categories
\begin{align*}
\nil(S;N) &\to \proj(S) \\
(P,p) & \mapsto P.\\
\end{align*}
Define groups $\widetilde{\nil}_*$ so that
$$
\nil_n(S;M) =\widetilde \nil_n(S;M) \oplus K_nS.
$$

We keep the notation from the introduction: $R$ is a ring,  $\go$ is a group sitting in a short exact
sequence
\[
1 \to F \to F \rtimes C_\infty = \go \xrightarrow{p} C_\infty \to 1
\]
and $\g= G_1 \ast_{F} G_2$ is a group sitting in a pushout diagram with injective maps
\[
\xymatrix{
F \ar[d] \ar[r] & G_2 \ar[d] \\
G_1 \ar[r] & G_1 \ast_{F} G_2 = \g.
         }
\]
We also choose $t \in \go$ so that $p(t) \in C_\infty$ is a generator.

However, in contrast  to the introduction, in the first half of this section
we do \emph{not} assume that $\go$ is a subgroup of $\g$ and we do \emph{not} assume that
$\g$ surjects onto $D_\infty$ or that  $F$ is of index $2$ in $G_1$ and $G_2$.

\begin{lemma} \label{dlnil}
Let $\ff$ be the smallest family of subgroups of $\g$ containing $G_1$ and $G_2$ and
let $\ff_0$ be the smallest family of subgroups of $\go$ containing $F$.
\begin{enumerate}
\item  \label{dlnil1}
The following exact sequences are split, and hence short exact.
\begin{align*}
H^{\g}_n(E_{\ff}\g ;\bfK_R) \to &H^{\g}_n(E_{\all}\g ;\bfK_R) \to H^{\g}_n(E_{\all}\g ,E_{\ff}\g ;\bfK_R)\\
H^{\go}_n(E_{\ff_0} \go ;\bfK_R) \to &H^{\go}_n(E_{\all}\go ;\bfK_R) \to H^{\go}_n(E_{\all}\go ,E_{\ff_0} \go ;\bfK_R)
\end{align*}
\item \label{dlnil2}
The relative terms can be expressed in terms of Nil groups; i.e., there are isomorphisms
\begin{align*}
H^{\g}_n(E_{\all} \g ,E_{\ff}\g ;\bfK_R) & \cong \widetilde{\WNil}_{n-1}(RF; R[G_1 - F], R[G_2 - F]) \\
H^{\go}_n(E_{\all} \go ,E_{\ff_0} \go ;\bfK_R) & \cong \widetilde{\Nil}_{n-1}(RF; tRF) \oplus  \widetilde{\Nil}_{n-1}(RF;t^{-1}RF)
\end{align*}

\end{enumerate}
\end{lemma}

The family $\ff_0$ is simply the  family $\sub(F)$ of all subgroups of $\go$ that are contained in $F$.  The family $\ff$ consists of
all subgroups of all conjugates of $G_1$ or $G_2$. Note that in the special case
considered in the introduction, where $p \colon \g \to D_\infty$ is a surjection onto the infinite dihedral group we have
\[
\ff = p^{\ast} \fin.
\]

Before we embark on the proof of Lemma~\ref{dlnil}, we review some facts about homotopy cartesian diagrams.  A square 
\[
\xymatrix{W \ar[d]^g \ar[r]^f & X \ar[d]^k \\
Y \ar[r]^\ell & Z}
\]
is {\em homotopy cartesian} if it is commutative and if the map $W \to X \times_Z Z^I \times_Z Y$ given by $f$ in the first factor, 
by the constant path $(x,t) \mapsto k(f(x))$ in the second factor, and by $g$ in the third, is a weak homotopy equivalence.  
A (possibly noncommutative) square is 
{\em homotopy cartesian with respect to a homotopy $H: W \times I \to Z$ from $k \circ f $ to $\ell \circ g$} 
if the map $W \to X \times_Z Z^I \times_Z Y$ given by $f$ in the first factor, by $H$ in the second, 
and by $g$ in the third, is a weak homotopy equivalence.  In that case the diagram 
\begin{equation}\label{hcartwrt}
\xymatrix{
W  \ar[rd] \ar[rrd]  \ar[rdd] & & \\
 &   X \times_Z Z^I \times_Z Y \ar[r] \ar[d]  &  X \ar[d] \\
 &    Z^I \times_Z Y \ar[r]   &  Z  
}
\end{equation}
is commutative, the square in the lower right is cartesian, and the horizontal maps are fibrations.
One can check that if

$$
\xymatrix{
W\ar[ddd] \ar[rrr] \ar[dr]&&& X \ar[ddd] \ar[dl]\\
&W' \ar[r] \ar[d]& X' \ar[d]& \\
& Y' \ar[r]& Z' & \\
Y \ar[rrr] \ar[ur] &&& Z \ar[ul]
}
$$
is a diagram with both squares commutative, with the inner square homotopy cartesian, with the diagonal arrows homotopy equivalences, and with all trapezoids homotopy commutative, then the outer diagram is also homotopy cartesian.

\begin{proof}[Proof of Lemma~\ref{dlnil}]
The proof of this lemma requires a careful comparison of the results of Waldhausen \cite{Waldhausen(1978a),Waldhausen(1978b)} and the
machinery of Davis-L\"uck \cite{DL98}.  We first review the main results of Waldhausen's paper.

Let $S \to A$ and $S \to B$ be {\em pure and free embeddings of rings}, i.e.~there are $S$-bimodule decompositions  $A = S \oplus M$ and $B = S \oplus N$  where $M$ and $N$ are free as right $S$-modules.  Let $A *_SB$ be the associated amalgamated free product of rings.

There is a (up to isomorphisms) commutative square of functors:

\begin{equation*}
\xymatrix{
\nil(S;M,N)  \ar[r]^-F \ar[d]^-G &  \proj(A) \times \proj(B) \ar[d]^-K  \\
 \proj(S)   \ar[r]^-L & \proj^*(A*_SB)}
\end{equation*}

Here
\begin{align*}
F(P,Q,p,q)  & = (A \otimes_S P, B \otimes_S Q) \\
G(P,Q,p,q) & =  P \oplus Q  \\
K(P,Q)  & =   (A*_SB) \otimes_A P \oplus (A*_SB) \otimes_B Q     \\
L(P) & =    (A*_SB) \otimes_S P
\end{align*}

  The exact category $\proj^*(A *_S B)$  is defined by Waldhausen \cite[III.10]{Waldhausen(1978b)}.  It is a  full subcategory of  $\proj(A *_S B)$  containing all finitely generated free modules; in particular its idempotent completion is equivalent to  $\proj(A *_S B)$.

In the formulas below, we write $\otimes$ instead of $\otimes_{S}$.  The above diagram commutes, but we ignore that and define an exact natural transformation $T : K \circ F  \to L\circ G$
$$
T(P,Q,p,q) = \begin{pmatrix}
1 & \widehat q \\
\widehat p & 1
\end{pmatrix} : ((A*_SB) \otimes P) \oplus ((A*_SB) \otimes Q) \to ((A*_SB) \otimes P) \oplus ((A*_SB) \otimes Q)
$$
Here $\widehat p$ is the composite of
$$
1_{A*_SB} \otimes p : (A*_SB) \otimes P \to  (A*_SB) \otimes M  \otimes Q
$$
and the map
$$
\text{mult}_{A*_SB}  \otimes 1_Q : ((A*_SB) \otimes M) \otimes Q \to (A*_SB) \otimes Q.
$$
Likewise for $\widehat q$.

This exact natural transformation gives a natural transformation $QT : Q(K \circ F) \to Q ( L\circ G )$ (here $Q$ denotes the Quillen $Q$-construction), hence a functor
$$
Q\nil(S;M,N) \times (0 \to 1) \to Q\proj^*(A*_SB)
$$
and hence a homotopy $H$ from $BQ ( K \circ F)$ to $BQ (L\circ G) $.  
Here $(0 \to 1)$ denotes the category with two objects 0 and 1 and three morphisms, including a morphism from 0 to 1. Recall that 
a functor $L: (0 \to 1) \times \calc \to \cald$ is the same as a natural transformation between functors $L_0, L_1  : \calc \to \cald$ and 
that $BL : B\left((0 \to 1) \times \calc\right) = I \times B\calc \to B \cald$ gives a homotopy from $BL_0$ to $BL_1$.

What Waldhausen actually proves (see \cite[Theorem 11.3]{Waldhausen(1978b)}) is that the diagram below is homotopy cartesian with respect to the homotopy $H$.

$$
\xymatrix{
BQ\nil(S;M,N)  \ar[r] \ar[d] & BQ\proj(A) \times BQ\proj(B)  \ar[d]  \\
  BQ\proj(S)   \ar[r] & BQ\proj^*(A*_SB)}
$$

This was promoted to a square of non-connective spectra by Bartels-L\"uck \cite[Theorem 10.2 and 10.6]{Bartels-Lueck(hk)} using the Gersten-Wagoner delooping of the $K$-theory of rings.  However, later on we will map from a square constructed using  Schlichting's non-connective spectra for the $K$-theory of exact categories (see \cite[Theorem 3.4]{Sch-exact} and  \cite{Sch-derived}), so we will use Schlichting's construction instead, noting that the Bartels-L\"uck-Gersten-Wagoner delooping of the above square is homotopy equivalent to that of Schlichting's, since Schlichting's delooping is a direct generalization of the Gersten-Wagoner delooping.  One obtains a commutative diagram of non-connective spectra:

\begin{equation} \label{twosidedsquare}
\xymatrix{
\bfnil(S;M,N)  \ar[r] \ar[d] &   \bfK(A) \vee \bfK(B) \ar[d]  \\
 \bfK(S)  \ar[r] & \bfK(A*_SB)}
\end{equation}
homotopy cartesian with respect to a homotopy ${\bf H} : \bfnil(S;M,N) \wedge I_+ \to \bfK(A*_SB)$.  (We have erased the superscript $*$ from the lower right, since the first step in the delooping is to replace the exact category by its idempotent completion, see \cite[Remark 3]{Sch-derived}).

We now turn to the one-sided case.  
Suppose $S$ is a subring of a ring $T$. Let $t \in T^\times$ be a unit so that $c_t(S) = t S t^{-1} = S$.  Then for any (left) $S$-module $P$, define an $S$-module $tP =\{tx : x \in P\}$ with $stx = tc_{t^{-1}}(s)x$ for $s \in S, x\in P$.  There are isomorphisms of $S$-modules
\begin{align*}
tP & \cong c_{t^{-1}}^*P   \cong c_{t*}P  \cong tS \otimes P\\
tx & \leftrightarrow \  x  \ \leftrightarrow \ 1 \otimes x \ \leftrightarrow \ t \otimes x.
\end{align*}
Note that if $P \subset T$ is an $S$-module, then the above definition coincides with the $S$-submodule $tP \subset T$.  There is a similar notation for $S$-bimodules.

Let $\alpha: S \to S$ be a ring automorphism.
Let $S_{\alpha}[t,t^{-1}] = \oplus_{-\infty}^{\infty}\, t^i S$ be the twisted Laurent polynomial ring with $st= t \alpha(s)$ for $s \in S$.  Note that $R[F \rtimes C_\infty] = S_\alpha[t,t^{-1}]$ where $S = RF$ and the automorphism $\alpha$ is induced by conjugation with $t^{-1}$.

There is a commutative square of functors:

\begin{equation}  \label{onesidedcatsquare}
\xymatrix{
\nil(S;tS) \amalg \nil(S; t^{-1}S )  \ar[r]^-{F = F_+ \amalg F_-}
\ar[d]^-{G = G_+ \amalg G_-} &  \proj(S)
\ar[d]^-K  \\
(0 \to 1) \times \proj(S)   \ar[r]^-L & \proj^*S_\alpha[t,t^{-1}]}
\end{equation}

Here
\begin{align*}
F_+(P,p)  = P, \quad & \quad F_-(Q,q)  = t^{-1}Q,\\
G_+(P,p)  = (0,P),  \quad & \quad G_-(Q,q)  = (1,Q),\\
K(P)   =   S_\alpha[t,t^{-1}] \otimes P,    \quad & \quad  \\
L_0(P)  = S_\alpha[t,t^{-1}] \otimes P,    \quad & \quad L_1(Q)  = S_\alpha[t,t^{-1}] \otimes t^{-1}Q,    \\
T_L(P) \colon  L_0(P) \to L_1(P) ; & \quad  \quad x \otimes y \mapsto xt \otimes t^{-1}y.
\end{align*}

The above diagram commutes, but we ignore that and define a natural transformation  $T = T_+ \amalg T_- : K \circ F \to L\circ G$
as follows.
\begin{align*}
T_+(P,p: P \to tP) = T_L(tP)\circ(1 \otimes p) &: S_\alpha[t,t^{-1}] \otimes P  \to S_\alpha[t,t^{-1}] \otimes P, \\
T_-(Q,q: Q \to t^{-1}Q) = (1 \otimes q)\circ T_L(Q)^{-1}&: S_\alpha[t,t^{-1}] \otimes t^{-1}Q  \to S_\alpha[t,t^{-1}] \otimes  t^{-1}Q.
\end{align*}

Waldhausen \cite[Theorem 12.3]{Waldhausen(1978b)} shows that the square \eqref{onesidedcatsquare}, after applying $BQ$, is homotopy cartesian with respect to the homotopy $H = BQT$.  Thus we obtain a
a square of non-connective spectra
\begin{equation} \label{onesidedsquare}
\xymatrix{ \bfnil(S;tS) \vee \bfnil(S; t^{-1}S) \ar[r] \ar[d] & \bfK(S) \ar[d]\\
 D^1_+ \wedge \bfK(S) \ar[r] & \bfK(S_\alpha[t,t^{-1}])
}
\end{equation}
which is homotopy cartesian with respect to the homotopy $\bfH$.

After our extensive discussion of the results of Waldhausen, we return to the situation of the lemma.
We treat the $\g$- and $\go$-case in parallel.
Define $E_{\ff} \g$ and $E_{\ff_0}\go$ as the  pushouts of $\g$-spaces, respectively $\go$-spaces
\begin{equation} \label{geomsquares}
\xymatrix{
S^0 \times \g/F  \ar[r]^-a \ar[d] & \g/G_1 \amalg \g/G_2 \ar[d]
& S^0 \times \go/F \ar[r]^-b \ar[d] & \go/F \ar[d]\\
D^1 \times \g/F \ar[r] & E_{\ff} \g & D^1 \times \go/F \ar[r] & E_{\ff_0} \go.
}
\end{equation}
Here $S^0= \{ -1, +1 \}$ and the upper horizontal ``attaching'' maps are given by
$a (+1, gF)  = gG_1$, $a(-1,gF)= gG_2$, $b( +1, gF) = g F$ and $b (-1, gF)  = g tF$.  Note that $E_{\ff} \g$ is a tree and $E_{\ff_0} \go$ can be
identified with the real line $\IR$ with the $\go$-action induced by the standard $C_\infty$-action.

Davis-L\"uck \cite{DL98} defined the $K$-theory $\Or G$-spectrum by first defining a functor $\bfK_R : \Groupoids \to \Spectra$, assigning to a groupoid $\mathcal G$ the Pedersen-Weibel spectrum of the $K$-theory of the additive category of finitely generated projective $R \mathcal G$-modules.  We instead use the Schlichting spectrum; Schlichting \cite[Section 8]{Sch-derived} shows that its homotopy type agrees with that of Pedersen-Weibel.   The functor $\bfK_R$ has the property that a natural transformation between maps of groupoids $f_0, f_1 : \calh \to \calh'$ induces a homotopy $\bfK_R(f_0) \simeq \bfK_R(f_1)$ and, in particular, equivalence groupoids have homotopy equivalent $K$-spectra (see Example \ref{homotopy_and_K} below.) 
A $G$-set $X$ gives a groupoid $\overline{X}$ whose objects are elements of $X$ and whose morphism set from $x_0 \in X$ to $x_1 \in X$  is given by $\{g \in G : gx_0 = x_1\}$ (composition is given by group multiplication).  As in \cite{DL98}, define an $\Or G$-spectrum $\bfK_R$ by defining $\bfK_R(G/H) =  \bfK_R (\overline{G/H})$.  There is a homotopy equivalence $\bfK_R ( G / H ) \simeq \bfK(RH)$ which follows from the equivalence of groupoids $\overline{H/H} \to \overline{G/H}$ induced from the inclusion.  
Indeed, let $\{g_i\}$ be a set of coset representatives for $G/H$.  Define a retract functor $\Phi : \overline{G/H} \to \overline{G/H}$ by $\Phi(g_iH) = H$ on objects and $\Phi(g : g_iH \to g_jH) = (g_j^{-1}g g_i : H \to H)$ on morphisms.  Then $g_i : \Phi(g_i H) \xrightarrow{\sim} g_i H$ is a natural isomorphism from $\Phi$ to the identity.

For a $G$-CW-complex $X$, set $\bfK_{R}(X) = \map_G(-,X)_+ \wedge_{\Or G} \bfK_R(-)$.
The spectrum $\bfK_{R} (X)$ is the spectrum whose homotopy groups were denoted $H_n^G(X;\bfK_R)$ in the introduction.   Note that the two definitions of $\bfK_R(G/H)$ in terms of groupoids and in terms of $G$-CW-complexes agree by ``Yoneda's Lemma.''

  Applying
$\bfK_{R}( -)$ with $G=\g$ to the left hand pushout square of \eqref{geomsquares} yields the inner square of the diagram:
\begin{small}
\begin{equation} \label{twosidedconsquares}
\xymatrix@R=8mm @C=3mm{
\bfK(RF) \vee \bfK(RF)\ar[ddd] \ar[rrr] \ar[dr]&&& \bfK(RG_1) \vee \bfK(RG_2) \ar[ddd] \ar[dl]\\
&\bfK_R(\g/F) \vee \bfK_R(\g/F) \ar[r] \ar[d]& \bfK_R(\g/G_1) \vee \bfK_R(\g/G_2) \ar[d]& \\
& D^1_+ \wedge \bfK_R(\g/F) \ar[r]& \bfK_{R}(E_{\ff} \g) & \\
D^1_+ \wedge \bfK(RF) \ar[rrr] \ar[ur] &&& \bfK_{R}(E_{\ff} \g) \ar@{=}[ul]
}
\end{equation}
\end{small}
Since $\bfK_{R}$ applied to a pushout square is homotopy cocartesian (\cite[Lemma 6.1]{DL98}) and since a homotopy cocartesian square of spectra is homotopy cartesian (\cite[Lemma 2.6]{LRV}), the inner square is homotopy cartesian.

The remaining maps in the above diagram should be reasonably clear, for example, the upper diagonal maps are induced by maps of groupoids $\overline{F/F} \to \overline{\g/F}$ and  $\overline{G_i/G_i} \to \overline{\g/G_i}$.  It is easy to see that the whole diagram commutes and  the diagonal maps are homotopy equivalences.  Thus the outer square is homotopy cartesian.

The inner square of 
\begin{small}
\begin{equation} \label{onesidedconsquares}
\xymatrix{
\bfK(RF) \vee \bfK(RF)\ar[ddd] \ar[rrr] \ar[dr]&&& \bfK(RF) \ar[ddd] \ar[dl]\\
&\bfK_R(\go/F) \vee \bfK_R(\go/F) \ar[r] \ar[d]& \bfK_R(\go/F) \ar[d]& \\
& D^1_+ \wedge \bfK_R(\go/F) \ar[r]& \bfK_{R}(E_{\ff_0} \go) & \\
D^1_+ \wedge \bfK(RF) \ar[rrr] \ar[ur] &&& \bfK_{R}(E_{\ff_0} \go) \ar@{=}[ul]
}
\end{equation}
\end{small}
is homotopy cartesian, since the right hand square of \eqref{geomsquares} is a pushout diagram.  The long horizontal map on the top is induced by the ring maps $\id \colon RF \to RF$ and $c_{t^{-1}} \colon RF \to RF$.  The definition of the rest of the maps in the above diagram should be reasonably clear. 
Everything commutes, except for the top trapezoid restricted to the second $\bfK(RF)$ factor.  We proceed to give an argument that this is homotopy commutative, using the fact that a natural transformation between maps of groupoids induces a homotopy between the associated maps in $K$-theory.  The upper diagonal maps are induced by the obvious embedding $\inc: \overline{F/F} \to \overline{\go/F}$.  The top horizontal map of the inner square is induced by the map of groupoids $\overline{R_t} : \overline{\go/F} \to \overline{\go/F}$ which is in turn induced by the $\go$-map $R_t : \go/F \to \go/F, \quad R_t(gF) = gtF$.  The upper horizontal map is induced by the map of groupoids $\overline{c_{t^{-1}}}: \overline{F/F} \to \overline{F/F}$ which sends a morphism $f : F \to F$ to the morphism $t^{-1}ft: F \to F$.  Finally, the natural transformation $L_t$ between the maps of groupoids $\inc \circ \overline{c_{t^{-1}}}, \overline{R_t} \circ \inc : \overline{F/F} \to \overline{\go/F}$ is given by the morphism $t : F \to tF$ in $\overline{\go/F}$.

Thus the upper trapezoid homotopy commutes.  The upper diagonal maps are homotopy equivalences since $\inc$ is an equivalence of groupoids.    Hence the outer square is also homotopy commutative.

The squares
\eqref{twosidedsquare} and \eqref{onesidedsquare} specialize to the squares below.

\begin{equation} \label{allsquares}
\xymatrix{
\bfnil(RF;R[G_1-F],R[G_2-F])  \ar[r] \ar[d] &   \bfK(RG_1) \vee \bfK(RG_2) \ar[d]  \\
 \bfK(RF)  \ar[r] & \bfK(R\g), \\
 \bfnil(RF;tRF) \vee \bfnil(RF; t^{-1}RF) \ar[r] \ar[d] & \bfK(RF) \ar[d]\\
 \bfK(RF) \ar[r] & \bfK(R\g_0).\\
}
\end{equation}

We next construct maps from the outer squares of \eqref{twosidedconsquares} and \eqref{onesidedconsquares} to the squares in \eqref{allsquares}.  In the upper left hand corners we use the fact that the Nil categories split off categories of projective modules.  In the upper right hand corner we use the identity.  In the lower left we use the map
$D^1_+ \wedge \bfK(RF) \to \pt_+ \wedge \bfK(RF) = \bfK(RF)$.  In the lower right we use the constant maps $E_{\ff}\g \to \g/\g$ and  $E_{\ff_0}\go \to \go/\go$.

These maps have two extra properties.  First, they are homotopy equivalences on the upper right and lower left hand corners. Second, the maps have the property that the homotopies $\bf H$ are constant on the image of the upper left hand corners of the outer squares.  This means that we have maps from the outer squares (which are homotopy cartesian) to the homotopy cartesian version of the squares in \eqref{allsquares}.  
Compare the discussion of \eqref{hcartwrt}. Hence we have maps from the Mayer-Vietoris exact sequences of the homotopy groups of the outer squares to those of the homotopy cartesian versions of \eqref{allsquares}.

Given a commutative diagram of abelian groups with exact rows
\[
\xymatrix{
\dots \ar[r] & D_n \ar[r] \ar[d] & A_{n-1} \ar[r] \ar[d] & B_{n-1} \oplus C_{n-1} \ar[r] \ar[d]^-{\cong} & D_{n-1} \ar[r] \ar[d] & \dots \\
\dots \ar[r] & D'_n \ar[r]        & A'_{n-1} \ar[r]        & B'_{n-1} \oplus C'_{n-1} \ar[r] & D'_{n-1} \ar[r]  & \dots
}
\]
a diagram chase gives a long exact sequence of abelian groups
\[
\xymatrix{
\dots \ar[r] & D_n \ar[r] & D'_n \oplus A_{n-1} \ar[r] & A'_{n-1} \ar[r] & D_{n-1} \ar[r] & \dots
}
\]

Applying this to the map from the homotopy exact sequences of \eqref{twosidedconsquares}, \eqref{onesidedconsquares}, and \eqref{allsquares}, we obtain exact sequences
\begin{multline}  \label{exactseqs}
\cdots \to H_{n}^\g(E_{\ff}\g;\bfK) \to H_{n}^\g(E_{\all}\g;\bfK) \oplus (K_{n-1}(RF) \oplus K_{n-1}(RF)) \\
\to \nil_{n-1}(RF; R[G_1-F],R[G_2-F]) \to \cdots
\end{multline}
\begin{multline*}
\cdots \to H_{n}^{\go}(E_{\ff_0}\go;\bfK) \to H_{n}^{\go}(E_{\all}\go;\bfK) \oplus (K_{n-1}(RF) \oplus K_{n-1}(RF)) \\
\to \nil_{n-1}(RF; tRF) \oplus \nil_{n-1}(RF;t^{-1}RF) \to \cdots
\end{multline*}

Canceling the copies of $K_{n-1}( RF)$, one obtains long exact sequences
\begin{equation}  \label{lesf}
\cdots \to H_{n}^\g(E_{\ff}\g;\bfK) \to H_{n}^\g(E_{\all}\g;\bfK)
\xrightarrow{\partial} \widetilde \nil_{n-1}(RF; R[G_1-F],R[G_2-F]) \to \cdots
\end{equation}
\begin{multline} \label{lesfo}
\cdots \to H_{n}^{\go}(E_{\ff_0}\go;\bfK) \to H_{n}^{\go}(E_{\all}\go;\bfK) \\
\xrightarrow{\partial} \widetilde \nil_{n-1}(RF; tRF) \oplus  \widetilde \nil_{n-1}(RF; t^{-1} RF) \to \cdots
\end{multline}

The maps labelled $\partial$ are precisely the same as the maps
\begin{align*}
K_nR\g &\to  \widetilde \nil_{n-1}(RF; R[G_1-F],R[G_2-F]) \\
K_nR\g_0 & \to \widetilde \nil_{n-1}(RF; tRF) \oplus  \widetilde \nil_{n-1}(RF;t^{-1} RF)
\end{align*}
which Waldhausen shows are split surjections in \cite[Theorem 11.6 and Theorem 12.6]{Waldhausen(1978b)}.
The Lemma follows.
\end{proof}

Our next goal will be a $C_2$-equivariant version of Lemma \ref{dlnil} (namely Lemma \ref{equivariant}) in the special situation of the introduction. This will be necessary for the proof of Theorem \ref{cor1}.  We need some preliminary definitions to define the $C_2$-action.  The exposition is similar to Brown's discussion of functoriality of homology of groups \cite[Chapter III, Section 8]{Brown}.  
Equivalently one could use L\"uck's notion of an equivariant homology theory \cite{Lueck-Chern}, with details in the Ph.~D.~thesis of J.~Sauer \cite{Sauer}.

\begin{definition}
Let $\gcw$ be the category whose objects are pairs $(G,X)$ where $G$ is a group and $X$ is a $G$-CW-complex.  A morphism $(\alpha, f) :(G,X) \to (G',X')$ is given by a homomorphism $\alpha: G \to G'$ and a $G$-map $f: X \to \alpha^*X'$.  Equivalently, $f$ is determined by a $G'$-map $\alpha_*X \to X'$.

Let $\Or \subset \gcw$ denote the full subcategory whose objects are of the form $(G , G/H)$ with $G$ a group and $H \subset G$ a subgroup. 
Note that for every fixed group $G$ the orbit category $\Or G$ is a subcategory of $\Or$ via the inclusion
$\inc_G \colon  \Or G \to \Or$ given by $G/H \mapsto (G, G/H)$ and
$f \mapsto (\id, f)$.
\end{definition}

Recall that a {\em groupoid} is a category whose morphisms are all invertible and that 
the morphisms in the category of groupoids are just functors. 
\begin{definition}

We denote by $\calg \colon \Or \to \Groupoids$ the following functor. An object $(G, G/H)$ gets mapped 
to the groupoid  whose set of objects is $G/H$ and
whose morphism sets are given by $\mor (gH, g'H) = \{ \gamma \in G \; | \; \gamma gH = g'H \}$. 
For a morphism $( \alpha ,f) \colon (G , G/H) \to (G' , G'/H')$ the corresponding functor 
between groupoids sends the object
$gH$ to $f(g)f(H)$ and the morphism $\gamma \in G$ with $\gamma gH = g'H$ to $\alpha (\gamma )$. Note that
in the notation used so far
\[
\calg \circ \inc_G ( G/H ) = \overline{G/H}.
\]
\end{definition}

\begin{definition}
Let $\bfE: \Groupoids \to \Spectra$ be a functor.  For each group $G$, let $\bfE^G : \OrG \to \Spectra$ be 
the composition 
\[
\Or G \xrightarrow{\inc_G} \Or \xrightarrow{\calg} \Groupoids \xrightarrow{\bfE} \Spectra.
\]
Hence 
$$
\bfE^G(G/H) = \bfE(\overline{G/H}).
$$
Define a functor
\[
\bfE : \gcw \to \Spectra
\]
 as follows.
On objects, let
$$
\bfE(G,X) = \map_G(-,X)_+ \wedge_{\OrG} \bfE^G(-).
$$
Every group homomorphism $\alpha \colon G \to G'$ induces a functor $\Or ( \alpha) \colon \Or G \to \Or G'$ 
that sends $G/H$ to $G'/\alpha(H)$ and the morphism $R_{\gamma}$ to $R_{\alpha(\gamma)}$. (Here recall if $\gamma \in G$  satisfies $\gamma^{-1}K\gamma \subset H$, then $R_\gamma^{}: G/K \to G/H; \quad gK \mapsto g\gamma H$ is a well-defined $G$-map; conversely any $G$-map $G/K \to G/H$ is of this form.)
Moreover there is a natural 
transformation $\mu( \alpha )$ from $\inc_G \colon \Or G \to \Or$ to $\inc_{G'} \circ \Or ( \alpha ) \colon \Or G \to \Or$ given
by $\mu(\alpha)_{G/H} = ( \alpha , \alpha_{G/H} ) \colon (G, G/H ) \to ( G' , G' / \alpha( H) )$ with $\alpha_{G/H}( gH) = 
\alpha(g) \alpha(H)$. If we compose with $\calg$ we will below often use the short notation
\[
\overline{\alpha} \colon \overline{ G/H} \to \overline{G/H}
\]
for 
\[
\calg ( \mu(\alpha)_{G/H} ) \colon \calg \circ \inc_G (G/H) \to \calg \circ \inc_G ( G/H).
\]

Every morphism $(\alpha,f): (G,X) \to (G',X')$ in $\gcw$ induces a natural transformation
$\sigma( f)$ from the contravariant functor $\map_G ( - , X) \colon \Or G \to \spaces$ to the contravariant functor
$\map_{G'} ( - , X') \circ \Or( \alpha) \colon \Or G \to \spaces$. Using the homeomorphism $\map_G ( G/H , X) \xrightarrow{\cong} X^H$, 
$\phi \mapsto \phi( eH)$ the transformation $\sigma(f)$ is given at the object $G/H$ by 
$\sigma(f)_{G/H} \colon X^H \to X'^{\alpha(H)}$, $x \mapsto f(x)$.

Summarizing we have functors and natural transformations 
\[
\xymatrix{
\spaces^{\op} \ar@{=}[d] & \ar@{=>}[dr]_-{\sigma ( f)} & &\ar[lll]_-{\map_G ( - ,X)} \ar[d] \Or G  \ar[d]_-{\Or(\alpha)} \ar[rrr]^-{\inc_G} & & \ar@{=>}[dl]^-{\mu (\alpha)} & \Or \ar@{=}[d]\\
\spaces^{\op} & & &\ar[lll]^-{\map_{G'} ( - ,X')} \Or G' \ar[rrr]_-{\inc_{G'}} & &  & \Or .
}
\]
Composing with $\bfE \circ \calg \colon \Or \to \Spectra$ on 
the right hand side we get a similar diagram with $\inc_G$, $\inc_{G'}$ and $\mu ( \alpha )$ replaced by
$\bfE^G$, $\bfE^{G'}$ and $\bfE \circ \calg ( \mu ( \alpha ) )$. This is exactly what is needed in order to induce a map
\[
\bfE( \alpha , f) \colon \bfE( G, X) \to \bfE (G' , X' )
\]
between the balanced smash products.  We have now defined $\bfE$ on morphisms.

There is a similar functor $\bfE : \gcwp \to \Spectra$ where an object $(G,(X,Y))$ of $\gcwp$ is a group $G$ and a $G$-CW-pair $(X,Y)$.  We omit the details of the definition in this case.
\end{definition}

Given a subgroup $H \subset G$, a $G$-space $X$, and an element $\gamma \in G$, define the $\gcw$-morphism
\begin{align*}
(c_{\gamma},\gamma \cdot): (H,\text{res}_H X) & \to (\gamma H\gamma^{-1},\text{res}_{\gamma H\gamma^{-1}} X)\\
(h,x) & \mapsto  (\gamma h\gamma^{-1},\gamma x),
\end{align*}
where $\text{res}_H X$ is the $G$-set $X$ with the action restricted to $H$.
Note $(c_e,e\cdot) = \id$ and $(c_\gamma,\gamma\cdot)) \circ (c_\gamma',\gamma'\cdot)  = (c_{\gamma\gamma'},\gamma \gamma' \cdot )$.  Thus if $H$ is normal in $G$, conjugation defines a $G$-action on the spectrum $\bfE(H,\res_H X)$, where $\gamma \in G$ acts by $\bfE(c_\gamma,\gamma\cdot)$.

We would like inner automorphisms to induce the identity, but we need another condition on our functor $\bfE$.

Let $( 0 \leftrightarrow 1)$ denote the groupoid with two distinct but isomorphic objects
and four morphisms.

\begin{definition} \label{h-inv}
A functor $\pi \colon \Groupoids \to \Ab$ from the category of groupoids to the category of abelian groups 
is called homotopy invariant if one of the following equivalent conditions is satisfied.
\begin{enumerate}
\item
For every groupoid $\calh$ the functor $i_0 \colon \calh \to \calh \times ( 0 \leftrightarrow 1 )$, 
which sends the object $x$ to $(x,0)$ induces an isomorphism $\pi(i_0)$.
\item
If there exists a natural transformation from the  functor $f_0 \colon \calh \to \calh'$ to $f_1 \colon \calh \to \calh'$, then 
$\pi( f_0 ) = \pi( f_1 )$.
\item
If $f \colon \calh \to \calh'$ is an equivalence of categories then $\pi(f)$ is an isomorphism.
\end{enumerate}
A functor $\bfE \colon \Groupoids \to \Spectra$ is called homotopy invariant 
if $\pi_n ( \bfE ( - ))$ is homotopy invariant for every $n \in \IZ$.
\end{definition}

\begin{proof}[Proof that these conditions are equivalent]
Let $p \colon \calh \times ( 0 \leftrightarrow 1) \to \calh$ denote the projection. If $\pi(i_0)$ is an isomorphism, then $p \circ i_0 = \id$
implies that $\pi(p)$ is its inverse. Note that the natural transformation between $f_0$ and $f_1$ is automatically 
objectwise an isomorphism and hence gives rise to a functor $h \colon \calh \times (0 \leftrightarrow 1 ) \to \calh'$ with
$f_0 = h \circ i_0$ and $f_1 = h \circ i_1$. Now $\pi(f_1) = \pi( h \circ i_1 ) = \pi( h) \pi( i_0 ) \pi(p) \pi(i_1) = \pi( h \circ i_0 ) = \pi( f_0)$
shows that (i) implies (ii). That (ii) implies (iii) follows straightforward from the definitions and since $i_0$ is an equivalence
of categories (i) follows from (iii).
\end{proof}

\begin{example}  \label{homotopy_and_K} 
For any ring $R$ and integer $i$, the functor $\pi_i\bfK_R : \Groupoids \to \Ab$ is homotopy invariant.  This is because a natural transformation between $f_0$ and $f_1$ induces an exact natural transformation between the associated exact functors on the categories of finitely generated projective modules.  In fact, the spectrum level maps $\bfK_R(f_0)$ and $\bfK_R(f_1)$ are homotopic.
\end{example}

For an $\Or G$-groupoid $F \colon \Or G \to \Groupoids$ we denote by $F \times (0 \leftrightarrow 1)$ the 
$\Or G$-groupoid with $(F \times (0 \leftrightarrow 1)) ( G/H ) = F( G/H) \times ( 0 \leftrightarrow 1 )$. As above
$i_0 \colon F \to  F \times ( 0 \leftrightarrow 1)$ denotes the obvious inclusion with $i_0 ( x) = (x , 0)$ where $x$ is an an object or morphism of $F(G/H)$.
\begin{lemma} \label{h-inv-sp}
Let $\bfE \colon \Groupoids \to \Spectra$ be a homotopy invariant functor and let $X \colon \Or G \to \spaces$ be a
contravariant functor which is a free $\Or G$-CW-complex, then the following holds.
\begin{enumerate}
\item
For every $\Or G$-groupoid $F \colon \Or G \to \Groupoids$ the inclusion $i_0 \colon F \to F \times (0 \leftrightarrow 1)$
induces an isomorphism
\[
\pi_{\ast} ( \id_{X_+} \sma \bfE (i_0) ) \colon
\pi_{\ast}( X_+ \sma_{\Or G} \bfE \circ F ) \to \pi_{\ast}( X_+ \sma_{\Or G} \bfE \circ ( F \times (0 \leftrightarrow 1)).
\]
\item
If $f_0 \colon F \to F'$ and $f_1 \colon F \to F'$ are maps of $\Or G$-groupoids such that there exists a homotopy between them, i.e.\ a map of $\Or G$-groupoids
\[
h \colon F \times (0 \leftrightarrow 1) \to F'
\]
with $h \circ i_0 = f_0$ and $h \circ i_1 = f_1$ then the induced maps coincide:
\[
\pi_{\ast} ( \id_{X_+} \sma \bfE ( f_0 ) ) = \pi_{\ast} ( \id_{X_+} \sma \bfE ( f_1) ).
\]
\item
If $f \colon F \to F'$ is a map of $\Or G$-groupoids, for which there exists a map $g \colon F' \to F$ of $\Or G$-groupoids
together with homotopies as defined in (ii) showing $g \circ f \simeq \id_F$ and $f \circ g \simeq \id_{F'}$ then the induced map
\[
\pi_{\ast} (\id_{X_+} \sma \bfE (f) ) \colon \pi_{\ast} ( X_+ \sma_{\Or G} \bfE \circ F ) \to 
\pi_{\ast} ( X_+ \sma_{\Or G} \bfE \circ F' )
\]
is an isomorphism.
\end{enumerate}
\end{lemma}
\begin{remark} \label{rem-spelled-out}
The existence of a homotopy between $f_0$ and $f_1$ as in (ii) means in concrete terms that for all $G/H \in \obj \Or G$
there exists a natural transformation $\tau (G/H)$ from $f_0(G/H) \colon F( G/H) \to F'(G/H)$ to
$f_1(G/H) \colon F(G/H) \to F'(G/H)$ such that the following condition is satisfied: for all objects $x \in \obj F(G/H)$ and all
$G$-maps $\alpha \colon G/H \to G/K$ the morphism $\tau (G/H)_x$ in $F'(G/H)$ is mapped under 
$F'( \alpha)$ to $\tau (G/K)_{F(\alpha)(x)}$, i.e.\
\[
F'( \alpha) ( \tau(G/H)_x) = \tau(G/K)_{F(\alpha)(x)}.
\]
\end{remark}
\begin{proof}[Proof of Lemma~\ref{h-inv-sp}]
Completely analogous to the proof of Definition~\ref{h-inv} above one shows that (i), (ii) and (iii) are equivalent. 
It hence suffices to show that (iii) holds. By Definition~\ref{h-inv}~(iii) the fact that $\bfE$ is homotopy invariant 
implies that for 
every $G/H$ the map
\[
\pi_{\ast}( \bfE ( f (G/H))) \colon \pi_{\ast} ( \bfE ( F(G/H))) \to \pi_{\ast} ( \bfE ( F'(G/H))) 
\]
is an isomorphism. Hence $\bfE (f)$ is a weak equivalence of $\Or G$-spectra. A well known argument 
(compare Theorem~3.11 in \cite{DL98})
using induction
over the skeleta of $X$ shows the claim.
\end{proof}

\begin{lemma}[Conjugation induces the identity] \label{cong=id}
Suppose $\bfE: \Groupoids \to \Spectra$ is a homotopy invariant functor.  Let $G$ be a group and $X$ a $G$-CW-complex.  Then for any $\gamma \in G$, 
$$
\bfE(c_\gamma,\gamma\cdot) : \bfE(G,X) \to \bfE(G,X)
$$ 
induces the identity on homotopy groups.  Thus the conjugation action of $G$ on $\pi_*\bfE(G,X)$ is the trivial action.
\end{lemma}

\begin{proof}
Let $T$ be the natural transformation from $\inc_G \colon \Or G \to \Or$ to itself that is given
at the object $G/H$ by 
\[
T_{G/H} = ( c_\gamma , L_{\gamma} ) \colon (G, G/H) \to (G, G/H),
\]
where $L_{\gamma} \colon G/H \to G/H$ maps $gH$ to $\gamma gH$. 
For every $G/H \in \obj\Or G$ there is a natural transformation $\tau ( G/H)$ from 
the identity $\id \colon (\calg \circ \inc_G)(G/H) \to (\calg \circ \inc_G) (G/H)$ to
$\calg (T_{G/H}) \colon (\calg \circ \inc_G)(G/H) \to (\calg \circ \inc_G)(G/H)$ given at the object $gH \in \obj \calg \circ \inc_G (G/H) = \overline{G/H}$ by the morphism $\gamma \colon gH \to \gamma gH$. One checks that the condition spelled out in
Remark~\ref{rem-spelled-out} is satisfied. Hence $\calg ( T) \colon \calg \circ \inc_G \to \calg \circ \inc_G$ 
and the identity $\id \colon \calg \circ \inc_G \to \calg \circ \inc_G$ are homotopic in the sense of 
Lemma~\ref{h-inv-sp}~(ii). According to Lemma~\ref{h-inv-sp}~(ii) the map 
\[
\id \sma \bfE \circ \calg (T) \colon \map_G ( - , X)_+ \sma_{\Or G} \bfE^G \to \map_G ( - , X)_+ \sma_{\Or G} \bfE^G
\] 
induces the identity on homotopy groups, because for a $G$-CW-complex $X$ the $\Or G$ space 
$\map_G( - , X)$ is a free $\Or G$-CW-complex.

We claim that 
\[
\bfE ( c_\gamma , \gamma \cdot ) = \id \sma \bfE \circ \calg ( T ).
\]
Indeed an element in the $n$th space of the spectrum
$\bfE(G,X)$ is represented by
$$
(x,y) \in X^H \times \bfE(\overline{G/H})_n
$$
where $H$ is a subgroup of $G$.  Then
\begin{align*}
\bfE(c_\gamma, \gamma \cdot))[x,y] & = [\gamma x, \bfE(\overline{c_\gamma})y] \\
& = [x, \bfE(\overline{R_\gamma}) \circ \bfE(\overline{c_\gamma})y].
\end{align*}
Here $\bfE ( \overline{R_{\gamma}} ) \circ \bfE (\overline{c_\gamma}) = 
 \bfE \left( \calg \left(\left(\id , R_{\gamma}\right) \circ \left(c\left(\gamma\right) , c_{\gamma_{G/H}}\right)\right)\right)$

A straightforward computation in $\Or$ shows that
\[
(\id , R_{\gamma}) \circ ( c_\gamma , c_{\gamma_{G/H} }) = ( c_\gamma, L_{\gamma} ) \colon (G , G/H) \to (G, G/H).
\]
Thus $\bfE( c_\gamma , \gamma \cdot ) = \id \wedge \bfE \circ \calg(T)$ which we have already shown induces the identity on homotopy groups.
\end{proof}



\begin{remark}  \label{conjugation}
Thus if $H$ is normal in $G$ and $X$ is a $G$-space, the conjugation action of $G$ on $\bfE(H,\res_H X)$ induces a $G/H$-action on $\pi_*\bfE(H, \res_H X)$.
\end{remark}

We now return to the group theoretic situation of the introduction.

\begin{lemma} \label{equivariant}
Let $p : \g \to D_\infty$ be an epimorphism, $C_\infty$ the maximal infinite cyclic subgroup of $D_\infty$, and $\go = p^{-1}C_\infty$.  Choose models for $E_{\all}\go$ and $E_{{\ff}_0}\go$ by restricting the $\g$-actions on $E_{\all}\g$ and $E_{{\ff}_0}\g$ to $\go$-actions.   By the above remark, conjugation induces a $C_2 = \g/\go$-action on $H_n^{\go}(E_{\all}\go, E_{{\ff}_0}\go;\bfK_R)$.  This, and  the isomorphism of Lemma \ref{dlnil}(ii),  induces a $C_2$-action on $\widetilde{\Nil}_{n-1}(RF; tRF) \oplus  \widetilde{\Nil}_{n-1}(RF;t^{-1}RF)$.  The  $C_2 = \g/\go$-action    switches the two summands.
\end{lemma}

\begin{proof}   A {\em weak homotopy $C_2$-spectrum} is a spectrum  $\bfE$ with a self map $T : \bfE \to \bfE$ so that $T \circ T$ induces the identity on homotopy groups.  A {\em weak homotopy $C_2$-map} between weak homotopy $C_2$-spectra $(\bfE,T)$ and $(\bfE',T')$ is a map of spectra $f : \bfE \to \bfE'$ so that $\pi_*(f \circ T) = \pi_*(T' \circ f)$.

We define weak homotopy $C_2$-actions on the outer and inner squares of \eqref{onesidedconsquares}
and on the bottom square of \eqref{allsquares} in such a way so that $C_2$ acts on the Mayer-Vietoris exact sequences  \eqref{exactseqs}, that the map between them is $C_2$-equivariant, and that $C_2$ switches the $\widetilde{\nil}$ summands.

There are basically three steps to the proof: to define the weak homotopy $C_2$-actions on the inner square of \eqref{onesidedconsquares}, the outer square of \eqref{onesidedconsquares}, and the bottom square of \eqref{allsquares}.

We accomplish the first step by constructing a $\g$-action on the inner square of \eqref{onesidedconsquares}.  First some notation.  Let
$$
D_\infty = \langle \sigma, \tau ~|~ \sigma^2 =1, \sigma \tau \sigma = \tau^{-1}\rangle
$$
and $C_\infty = \langle \tau \rangle$.  Choose $s,t \in \g$ so that $p(s) = \sigma$ and $p(t) = \tau$.  Let $F = \ker p$ and $\go = p^{-1}C_\infty$.  Then $\go = F \rtimes C_\infty$.  Note $\go/F = \{t^i F : i \in \Z\}$.

We first extend the $\go$-action on the right hand square of \eqref{geomsquares} to a $\g$-action.  Indeed, consider the square of $\g$-spaces
\begin{equation} \label{gammasquare}
\xymatrix{\g/F \ar[r]^-\Phi \ar[d] & \g/\langle t^{-1}s, F\rangle  \ar[d]\\
M\ar[r]& P}
\end{equation}
Let $\Phi(gF) = g\langle t^{-1}s,F\rangle$, $\pi: \g/F \to \g/\langle s,F\rangle$ with $\pi(gF) = g \langle s,F\rangle$,   $M$ be the mapping cylinder $M(\pi) = ([0,1] \times \g/F) \amalg \g/\langle s,F \rangle)/(0,gF) \sim \pi(gF)$, and $P$ be the pushout of the rest of the diagram.   (The $\g$-space $P$ can be identified with the $\g$-space $\R$, where the action is given via $\g \to D_\infty = \text{Isom}(\Z) \hookrightarrow \text{Isom}( \R)$.)

A bijection of $\go$-squares from the right hand square of \eqref{geomsquares} to the restriction of 
\eqref{gammasquare} is provided by
\begin{align*}
\psi_0: \go/F & \to \g/\langle t^{-1}s,F \rangle \\
t^iF &\mapsto t^i\langle t^{-1}s,F\rangle \\
\psi_S : S^0 \times \go/F & \to \g/F\\
(1,t^iF) & \mapsto t^iF \\
(-1,t^iF) & \mapsto t^isF\\
\psi_D :  D^1 \times \go/F & \to M\\
\psi_D(x,t^iF) & = \begin{cases}
[x,t^iF] & x \geq 0\\
[-x,t^isF] & x \leq 0
\end{cases}
\end{align*}
and mapping using the pushout property on the lower right hand corner.  

Via this bijection there is a $\Gamma$-action  on the right hand square of \eqref{geomsquares} which extends the $\Gamma_0$-action.  Formulae for this action are determined by knowing how $s \in \g - \go$ acts on $F$, $\{\pm 1\} \times F$, and $D^1 \times F$.  
 Note $s = t(t^{-1}s)$, so $s\cdot F = \psi_0^{-1}(s\psi_0(F)) = \psi_0^{-1}\left(s\langle t^{-1}s,F\rangle\right) = tF$. 
 Likewise  $s\cdot (\pm 1, F) = (\mp 1, F)$ and $s \cdot (x,F) = (-x,F)$ for $x \in D^1$.

By Remark \ref{conjugation} we have constructed a $\Gamma$-action on the inner square of \eqref{onesidedconsquares}, which gives a $C_2 = \g/\go$-action after applying homotopy groups. The action of $s \in \Gamma$ gives the weak homotopy $C_2$-action on the inner square of \eqref{onesidedconsquares}.  We have thus completed our first step.  

For future reference, we now examine how these weak homotopy $C_2$-actions are induced by maps of groupoids.  We start with the upper right hand corner and find formulae for $C_2$-action on the target of the Yoneda homeomorphism
\begin{align*}
\bfK_R(\go/F) =\map_{\go}(-, \go/F)_+ \wedge_{\Or \go} \bfK_R(-) &\xrightarrow{\cong} \bfK_R(\overline{\go/F})\\
[f,y] & \mapsto \bfK_R(f)y.
\end{align*}

Using the notation of the proof of Lemma \ref{cong=id}, 
\begin{align*}
\bfK_{R}(c_s, s \cdot )[F,y] & = [tF, \bfK_R(\overline{c_s})y] \\
& =  [F, \bfK_R(\overline{R_t} \circ \overline{c_s})y]
\end{align*}
Thus the $C_2$-action is given by $\bfK_R(\overline{R_t} \circ \overline{c_s}) : \bfK_R(\overline{\go/F}) \to \bfK_R(\overline{\go/F})$.  Furthermore, there is a natural transformation $L_t$ between the maps of groupoids $\overline{c_{t^{-1}s}}, \overline{R_t} \circ \overline{c_s}: \overline{\go/F} \to \overline{\go/F}$, and thus the $C_2$-action is given on homotopy groups by conjugation by $t^{-1}s$.

Similarly there are Yoneda homeomorphisms
\begin{align*}
\bfK_R(S^0 \times \go/F) =\map_{\go}(-,S^0 \times \go/F)_+ \wedge_{\Or \go} \bfK_R(-) &\xrightarrow{\cong} \bfK_R(\overline{S^0 \times \go/F}) \\
&= S^0_+ \wedge \bfK_R(\overline{\go/F})\\
\bfK_R(D^1 \times \go/F) =\map_{\go}(-,D^1 \times  \go/F)_+ \wedge_{\Or \go} \bfK_R(-) &\xrightarrow{\cong} D^1_+ \wedge \bfK_R(\overline{\go/F})\\
\end{align*}
The weak homotopy $C_2$-action on $\bfK_R(\overline{S^0 \times \go/F})$ is induced by the self-map of groupoids $\overline{-1 \times c_s}$.  The weak homotopy $C_2$-actions on  $S^0_+ \wedge \bfK_R(\overline{\go/F})$ and $D^1_+ \wedge \bfK_R(\overline{\go/F})$ are given by $(-1) \wedge \bfK_R(\overline{c_s})$.

The second step is to define weak homotopy $C_2$-actions on the outer square of \eqref{onesidedconsquares} so that the diagonal homotopy equivalences in \eqref{onesidedconsquares} are  weak homotopy $C_2$-maps.  This involves defining  weak homotopy $C_2$-actions  on three vertices (the action on $\bfK_R(E_{\ff}\go)$ was already defined in step 1) and verifying that four non-identity maps are  weak homotopy $C_2$-maps.

The weak homotopy $C_2$-actions are given by $\bfK(c_{t^{-1}s})$ on $\bfK(RF)$, by the switch map composed with $\bfK(c_s) \vee \bfK(c_s)$ on $\bfK(RF) \vee \bfK(RF)$ and by $[x,y] \mapsto [-x, \bfK(c_s)]$ on $D^1_+ \wedge \bfK(RF)$.  These are weak homotopy $C_2$-actions since  $s^2, (t^{-1}s)^2 \in F$ and inner automorphisms induce maps homotopic to the identity on $K$-theory.  (Indeed, note that for $f \in F$, the morphism $f : F \to F$ gives a natural transformation between $\id, \overline{c_f} : \overline{F/F} \to \overline{F/F}$.  Thus $\bfK(c_f)$ is homotopic to the identity.)

Inspection shows that all maps commute with the $C_2$-actions except for the top horizontal map and the upper right diagonal maps in \eqref{onesidedconsquares}.  The top horizontal map is induced by the map of groupoids
$$
\id \amalg \overline{c_{t^{-1}}} : (\overline{F/F})_+ \amalg (\overline{F/F})_- \to \overline{F/F}.
$$
Note $\overline{c_{t^{-1}}} \circ \overline{c_s} = \overline{c_{t^{-1}s}} \circ \id : (\overline{F/F})_+ \to \overline{F/F}$, while $\id \circ \overline{c_s}, \overline{c_{t^{-1}s}} \circ \overline{c_{t^{-1}}} : (\overline{F/F})_- \to \overline{F/F}$ are related by the natural transformation $f : F \to F$ where $t^{-1}st^{-1} = fs$.  Thus the top horizontal map is a weak homotopy $C_2$-map.  Now we examine the diagonal map in the upper right induced by the inclusion $\inc : \overline{F/F} \to \overline{\go/F}$.  Here $\inc \circ  \overline{c_{t^{-1}s}}, (\overline{R_t} \circ \overline{c_s}) \circ \inc : \overline{F/F} \to \overline{\go/F}$ are related by the natural transformation $L_t$.

  We have now completed our second step and have constructed a $C_2$-action up to homotopy on the outer square of \eqref{onesidedconsquares}.


The third step is to define a $C_2$-action on the bottom square of \eqref{allsquares}, compatible with the homotopy $\bfH$ and the weak homotopy $C_2$-action constructed on the outer square of \eqref{onesidedconsquares}.  The key difficulty will be defining the $C_2$-action on the Nil-term.  We take a roundabout route.  We will first define a $\go$-action, extend it to a $\g$-action, note that $F$-acts homotopically trivially, and then restrict to $C_2 = \langle s , F \rangle/F$.

An element $g \in \go$ acts on $\nil(RF; t^{\pm 1}RF)$ by
$$
g_\cdot(P,p: P \to t^{\pm 1}P) = (c_{g*}P,    t^{\pm 1}c_{g*}(\ell_{f,P}) \circ c_{g*}(p):  c_{g*}P \to t^{\pm 1} c_{g*}P)
$$
where $f = g^{-1}t^{\mp 1}gt^{\pm 1} \in F$, and $\ell_{f,P}(P) : c_{f*}P \to P$ is the homomorphism
$$
\ell_{f,P}\left( \sum s_i \otimes x_i\right) = \sum  s_i f x_i.
$$
Note that $\ell_{f,-}$ is a natural transformation from $c_{g*}$ to $\id$.  In making sense of the composite we used  the identifications
$$
c_{g*}(t^{\pm 1}P) = c_{gt^{\pm 1}*}P  = c_{t^{\pm 1}gf*}P =  t^{\pm 1}c_{g*}(c_{f*}P).
$$

The $\go$-action on $\nil(RF; t^{\pm 1}RF)$ extends to a $\g$-action acts on
$$\nil(RF; tRF)  \amalg \nil(RF; t^{-1}RF)
$$
as follows.   For $g \in \g - \go$, define
by
$$
g_\cdot(P,p: P \to t^{\pm 1}P) = (c_{g*}P,    t^{\mp 1}c_{g*}(\ell_{f,P}) \circ c_{g*}(p): c_{g*}P \to t^{\mp 1} c_{g*}P)
$$
where $f = g^{-1}t^{\pm 1} gt^{\pm 1} \in F$.
In making sense of the composite we used  the identifications
$$
c_{g*}(t^{\pm 1}P) = c_{gt^{\pm 1}*}P  = c_{t^{\mp 1}gf*}P =  t^{\mp 1}c_{g*}(c_{f*}P).
$$

We next claim that if $g\in F$ then $g$ acts trivially by showing that there is an exact natural transformation from $g : \nil(RF; t^{\pm 1}RF) \to \nil(RF; t^{\pm 1}RF)$  to the identity.  In fact, it is given by the pair of maps $\ell_{g,P} : c_{g*}P \to P$ and $t^{\pm 1}\ell_{g,P} : t^{\pm 1}c_{g*}P \to t^{\pm 1}P$.

In particular $s^2 \in F$ acts homotopically trivially, so the action by $s \in \g$ gives the homotopy $C_2$-action.  Notice that $s$ switches the two Nil-terms.  To complete the homotopy $C_2$-action on the bottom square of \eqref{allsquares}, let $C_2$ act via the ring automorphism $c_s: RF \to RF$ on the lower left, by $c_{t^{-1}s}: RF \to RF$ on the upper right, and trivially on the lower left.  It is easy to see that the previously constructed map from the upper left hand corner of outer square of \eqref{onesidedconsquares} to the upper left hand corner of the 
bottom square of \eqref{allsquares} are compatible with the $C_2$-action, as well has the homotopy $\bfH$.  It follows that all maps in \eqref{onesidedconsquares} are weak homotopy $C_2$-maps.  

Thus we have a $C_2$-equivariant map between the Mayer-Vietoris sequences \eqref{exactseqs} which switches the Nil-terms.  The Lemma follows.
\end{proof}

%


%



\section{Proof of Theorem \ref{cor1}}  \label{sec-proof-of-cor1}

The proof will require a sequence of lemmas, the most substantial of which is the following.

\begin{lemma} \label{spectra}
Let $\ff$ be a family of subgroups of $G$.
Let $\bfE : \Or G \to \Spectra $ be an $\Or G\text{-spectrum}$.  Then there is a homotopy cofiber sequence of $\Or G\text{-spectra}$
$$
\bfE_{\ff} \to \bfE \to \bfE/\bfE_{\ff}
$$
satisfying the following properties.
\begin{enumerate}
\item \label{coeff}
For any subgroup $H$ of $G$, one can identify the change of spectra map
$$
H_n^G(G/H; \bfE_{\ff}) \to H_n^G(G/H; \bfE)
$$
with  the change of space map
$$
H^H_n(E_{\ff \cap H}H;\bfE) \to H^H_n(\pt;\bfE) = \pi_n\bfE(H/H).
$$
Here $\bfE : \Or H \ \to \Spectra$ is defined by restriction: $\bfE(H/K) = \bfE(G/K)$.  
\item \label{cspectra}  For any family $\famh$ contained in $\ff$, the map $H^G_n(E_{\famh}G; \bfE_{\ff}) \to H^G_n(E_{\famh}G; \bfE)$ is an isomorphism.
\item  \label{fcoeff}
For any family $\famg$ containing $\ff$, the map
$
H^G_n(E_{\ff}G; \bfE_{\ff}) \to H^G_n(E_{\famg}G; \bfE_{\ff})
$
is an isomorphism.
\end{enumerate}
\end{lemma}

\begin{proof}  [Proof of Lemma \ref{spectra}] Define the $\Or G$-spectrum $\bfE_{\ff}$ by
$$\bfE_{\ff}(G/H) = \map_G(-, G/H \times E_{\ff}G)_+ \wedge_{\Or G} \bfE(-).$$ 
Projection $G/H \times  E_{\ff}G \to G/H$ is a  $G$-map  and hence gives a map of $\Or G$-spectra $\bfE_{\ff} \to \bfE$; define $ \bfE_{\ff}/\bfE$ to be the homotopy cofiber.

The proof of (i) will follow from the following commutative diagram,
where the indicated maps are isomorphisms.
\begin{equation} \label{lemma_diagram}
\xymatrix{  
H_n^G(G/H; \bfE_{\ff} ) \ar[rdd] & \\
H^G_n(G/H \times E_{\ff}G; \bfE) \ar[rd] \ar[u]_{\cong}   & \\
H^G_n(G \times_H \res_H E_{\ff} G ; \bfE) \ar[r] \ar[u]_{\cong}  &  H^G_n(G/H; \bfE) \\
H^H_n(\res_H E_{\ff}G;\bfE)  \ar[r] \ar[u]_{\cong} & H^H_n(\pt; \bfE) \ar[u]_{\cong} 
}
\end{equation}


The first horizontal map is induced from the map of spectra
$\bfE_{\ff} \to \bfE$, the other horizontal maps are induced
from the projection $\res E_{\ff} G = E_{\ff}G \to \pt$.
The top vertical isomorphism comes from a lemma of Bartels \cite[Lemma 2.3]{Bartels} which states that for any $G$-space $X$ there is an isomorphism
\begin{equation} \label{bartelsiso}
 H^G_*(X \times E_{\ff}G; \bfE)  \xrightarrow{\cong}  H^G_*(X; \bfE_{\ff}). 
\end{equation}
To see that the upper triangle commutes, it is helpful to write the homeomorphism of spectra underlying Bartel's isomorphism:
\begin{multline*}
 \map_G(?, X \times E_{\ff}G)_+ \wedge_{\Or G} \bfE(?)  \xrightarrow{\cong}
\map_G(?, X)_+ \wedge_{\Or G} \bfE_{\ff}(?)=  \\
  \map_G(?, X)_+ \wedge_{\Or G} \left(\map_G(??,? \times E_{\ff}G)_+ 
\wedge_{\Or G} \bfE(??) \right)
\end{multline*}
with formula $[\alpha,x] \mapsto [\alpha_1, [(\id_?, \alpha_2),x]]$ where $\alpha = (\alpha_1, \alpha_2) \colon ? \to X \times E_{\ff}G$ and $x\in E_n(?)$.

For any $G$-space $X$ there is a  commutative triangle of $G$-spaces
$$
\xymatrix{
G/H \times X \ar[rd] & \\
G \times_H \res_H X \ar[r] \ar[u]_{\cong} & G/H
}
$$
with the vertical homeomorphism given by $[g,x] \mapsto (gH,gx)$. The
middle triangle in \eqref{lemma_diagram} is induced from this triangle.

Next we discuss the bottom square of diagram \eqref{lemma_diagram}.
One of the axioms in \cite{Lueck-Chern} of an equivariant (as opposed
to $G$-equivariant) homology theory $\calh$ is the existence of
induction isomorphisms 
$\calh_*^H(X) \xrightarrow{\cong} \calh_*^G(G \times_H X)$ that are
natural in $X$.   
Proposition~157 of \cite{Lueck-Reich(survey)} (see also \cite{Sauer}) shows that this axiom is satisfied for $H^G_*(-; \bfE)$.

For the proof of (ii), note that for any family $\famh \subset \ff$, the spaces $E_{\famh} G \times E_{\ff} G $ and $E_{\famh} G $ are $G$-homotopy equivalent, hence 
\begin{align*}
H^G_n(E_{\famh}G; \bfE_{\ff}) & \cong  H^G_n(E_{\famh}G\times
E_{\ff}G; \bfE)  \quad \mbox{by \eqref{bartelsiso}}\\
  & \cong  H^G_n(E_{\famh}G; \bfE).
\end{align*}
For the proof of (iii), note for any family $\famg \supset \ff$, the spaces $E_{\ff}G \times E_{\ff}G$, $E_{\ff}G$, and $E_{\famg}G \times E_{\ff}G$ are $G$-homotopy equivalent, hence
\begin{align*}
H^G_n(E_{\ff}G; \bfE_{\ff}) &  \cong  H^G_n(E_{\ff}G \times E_{\ff}G; \bfE)  \quad \mbox{by \eqref{bartelsiso}}\\
 & \cong H^G_n(E_{\famg}G \times E_{\ff}G ; \bfE)\\
 & \cong  H^G_n(E_{\famg}G; \bfE_{\ff}) \quad \mbox{again by \eqref{bartelsiso}}.
\end{align*}
\end{proof}

Lemma \ref{spectra} shows that for $\ff \subset \gggg$ 
the change of space map
$$
H^G_n(E_{\ff}G; \bfE) \to H^G_n(E_{\famg}G; \bfE)
$$
can be identified with the change of spectra map
$$
H^G_n(E_{\famg}G; \bfE_{\ff}) \to H^G_n(E_{\famg}G; \bfE)
$$
since the domain of each is isomorphic to $H^G_n(E_{\ff}G;\bfE_{\ff})$.

In preparation for the proof of Theorem~\ref{cor1} we state two more lemmas.
\begin{lemma} \label{lem-join}
If $\famg$ and $\famh$ are families of subgroups of $G$ then there is a homotopy pushout square of
$G$-spaces
\[
\xymatrix{  E_{\famg \cap \famh} G \ar[d] \ar[r] & E_{\famg} G \ar[d] \\
E_{\famh} G \ar[r] & E_{\famg \cup \famh} G
}
\]
\end{lemma}
\begin{proof} Use the double mapping cylinder model for the homotopy pushout and verify that it satisfies
the characterizing property for $E_{\famg \cup \famh} G$.
\end{proof}
\begin{remark} \label{rem-join}
Note that $E_{\famh}G \xleftarrow{p_1} E_{\famh}G \times E_{\famg} G \xrightarrow{p_2} E_{\famg}G$
(where $p_1$ and $p_2$ are the obvious projections) is a model for
$E_{\famh}G \leftarrow E_{\famg \cap \famh}G \rightarrow E_{\famg}G$ and that the homotopy pushout of
$X \xleftarrow{p_1} X \times Y \xrightarrow{p_2} Y$ is the join $X \ast Y$.
This gives a conceptional explanation of the join-model $S^{\infty} \ast \IR$ of
$E_{\fbc} D_\infty = E_{\sub(C_\infty) \cup \fin} D_\infty$ discussed in Example~\ref{ex-join-model}.
\end{remark}


\begin{lemma} \label{lem-sub}
Suppose $N \triangleleft G$ is a normal subgroup of $G$ and $\bfE$ is an $\Or G$-spectrum. Then
there is a weak equivalence
\[
E (G/N)_+ \sma_{G/N} \bfE ( G/N ) \xrightarrow{\sim}
E_{\sub (N)}G_+ \sma_{\Or G} \bfE .
\]
Here on the left hand side $G/N$ acts on $\bfE ( G / N)$ via the identification
$G/N = \aut_{\Or G} ( G/N)$.
\end{lemma}
\begin{proof}
Note that the full subcategory of $\Or (G, \sub(N))$ on the single object $G/N$ is $\aut_{\Or G} ( G/N )$.  Let
$j_H \colon \aut_{\Or G} ( G/N ) \to \Or (G , \sub (N))$ denote the inclusion functor.
The map in question can be interpreted as the natural map
\[
\hocolim_{\aut_{\Or G}(G/N)} j_H^{\ast} \bfE \to \hocolim_{\Or (G , \sub(N)} \bfE.
\]
For every object $G/H$ in $\Or(G , \sub(N))$ the overcategory $G/H \downarrow j_H$ has the map
$G/H \to G/N$, $gH \mapsto gN$ as a final object, because both $\mor(G/H , G/N)$ and $\mor(G/N, G/N)$ can
be identified with $G/N$. Hence $j_H$ is right-cofinal and the map is a weak equivalence,
compare \cite[Proposition~4.4]{Hollender-Vogt}.
\end{proof}

\begin{remark}
Note that a model for $E_{\sub(N)} G$ is given by
$E(G/N)$ considered as a $G$-space via the projection $G \to G/N$.
\end{remark}

\begin{proof}[Proof of Theorem~\ref{cor1}]
Recall the notation concerning the groups from the introduction:
\[
p \colon \g = G_1 \ast_{F} G_2 \to D_\infty = C_2 \ast C_2,\  \go= p^{-1} (C_\infty), \  \text{ and } \  \g / \go \cong C_2.  
\]

Furthermore we set $\ff = p^*\fin$. Note that $\ff \cup \sub( \g_0 ) = p^{\ast} \fbc$.
We will abbreviate $\bfK_R$ by $\bfK$.  The proof of Theorem~\ref{cor1} consists in the following sequence of isomorphisms.  (We put an arrow on each isomorphism, indicating the easier direction to define.)
\begin{align*}
\widetilde{\WNil}_{n-1}   (RF; R[G_1-F] , R[G_2 -F] ) &\\
 \xleftarrow{\cong} H_n^{\g} ( E_{\all} \g , E_{\ff} \g ; \bfK )   & \quad \text{by Lemma~\ref{dlnil}\ref{dlnil2}} \\
 \xleftarrow{\cong}  H_n^{\g} ( E_{p^{\ast} \fbc} \g , E_{\ff} \g ; \bfK)  \label{eq2}  & \quad \text{by Theorem~\ref{thm-main}}\\
  \xrightarrow{\cong}  H_n^{\g} ( E_{p^{\ast} \fbc} \g , E_{\ff} \g ; \bfK/\bfK_{\ff})    & \quad \text{by Lemma~\ref{spectra}\ref{fcoeff}}\\
 \xleftarrow{\cong}  H_n^{\g} ( E_{p^{\ast} \fbc } \g ; \bfK/\bfK_{\ff} )&  \quad \text{by Lemma~\ref{spectra}\ref{cspectra}} \\
 \xleftarrow{\cong}  H_n^{\g} ( E_{\sub \go} \g ;\bfK/\bfK_{\ff}) & \quad \text{by Lemma~\ref{lem-join} with $\famg = \sub(\go)$,}\\
 & \quad \text{$\famh = \ff$ and Lemma~\ref{spectra}\ref{cspectra}}\\
 \xleftarrow{\cong}  H_n^{C_2} ( EC_2 ; (\bfK/\bfK_{\ff}) ( \g / \go))  & \quad \text{by Lemma~\ref{lem-sub}}\\
 \cong  H_0^{C_2} ( EC_2 ; \pi_n ( (\bfK/\bfK_{\ff}) ( \g / \go )) )   & \quad \text{see below}\\
 \cong
\left( \widetilde{\Nil}_{n-1} ( RF ; tRF )
\oplus \widetilde{\Nil}_{n-1} ( RF ; t^{-1} RF) \right)_{C_2}   & \quad \text{see below}\\
= \widetilde{\Nil}_{n-1} ( RF ; t RF ) &
\end{align*}

There is an Atiyah-Hirzebruch spectral sequence (derived from a skeletal filtration of $EC_2$) with 
\[
E^2_{p,q} = H_p^{C_2} ( EC_2 ; \pi_q ( (\bfK/\bfK_{\ff})( \g / \go )) )
\Rightarrow H_{p+q}^{C_2} ( EC_2 ; (\bfK/\bfK_{\ff}) (\g / \go) ).
\]
By Lemma \ref{spectra}\ref{coeff},
$$
\pi_q ((\bfK/\bfK_{\ff}) ( \g / \go ) ) \cong H_q^{\go}(E_{\all}\go, E_{{\ff}_0}\go;\bfK).
$$
Here we write $\ff_0 = \ff \cap \Gamma_0$.
By Lemma \ref{equivariant},
\[
 H_q^{\go}(E_{\all}\go, E_{{\ff}_0}\go;\bfK) \cong \widetilde{\Nil}_{q-1} (RF ; tRF) \oplus \widetilde{\Nil}_{q-1} (RF ; t^{-1} RF).
\]
Furthermore the $C_2$-action
interchanges the two summands. Hence $E^2_{p,q} = 0$ for $p > 0$, the spectral sequence collapses, and $E^2_{0,q}$ is given by the $C_2$-coinvariants. 
\end{proof}

\end{document}